\documentclass{article}
\usepackage[margin=0.5in]{geometry}
\usepackage[utf8]{inputenc}
\usepackage{fullpage}
\usepackage{enumerate}
\usepackage{authblk}
\usepackage[colorinlistoftodos]{todonotes}
\usepackage{verbatim}%comments out anything between \begin{comment} and \end{comment}
\usepackage{section, amsthm, textcase, setspace, amssymb, lineno, amsmath, amssymb, amsfonts, latexsym, fancyhdr, longtable, ulem}
\usepackage{epsfig, graphicx, pstricks,pst-grad,pst-text,tikz, colortbl}
\usepackage{epsf}
\usepackage{graphicx, color}
\usepackage{float}
\usepackage[rflt]{floatflt}
\usepackage{amsfonts, mathrsfs}
\usepackage{latexsym}
\usetikzlibrary{fit,matrix,positioning}
\usepackage{pdflscape}
\usepackage{kbordermatrix, enumitem}
\usepackage[most]{tcolorbox}
\usepackage{faktor}

\newtheorem{theorem}{Theorem}
\newtheorem{lemma}[theorem]{Lemma}
\newtheorem{proposition}[theorem]{Proposition}
\newtheorem{definition}[theorem]{Definition}
\newtheorem{Ex}[theorem]{Example}
\newtheorem{remark}[theorem]{Remark}

\newtheorem*{theorem*}{Theorem}
\newtheorem{conj}{Conjecture}

\newcommand{\ind}{{\rm ind \hspace{.1cm}}}

\newcommand{\rn}[1]{{\color{red} #1}}

\newcommand\addvmargin[1]{
  \node[fit=(current bounding box),inner ysep=#1,inner xsep=0]{};}

\begin{document}

\title{Contact Lie poset algebras of types B, C, and D}

\author[*]{Nicholas Mayers}
\author[**]{Nicholas Russoniello}

\affil[*]{Department of Mathematics, North Carolina State University, Raleigh, NC, 27605}
\affil[**]{Department of Mathematics, College of William \& Mary, Williamsburg, VA, 23187}

\maketitle

\bigskip
\begin{abstract} 
\noindent
We extend a recently established combinatorial index formula applying to Lie poset algebras of types B, C, and D. Then, using the extended index formula, we determine a characterization of contact Lie poset algebras of types B, C, and D corresponding to posets of height one in terms of an associated graph.
\end{abstract}

%\linenumbers

%%%%%%%%%%%%%%%%%%%%%%%%%%%%%%%%%%%
%%%%%%%%%%%%%%%%%%%%%%%%%%%%%%%%%%%
\section{Introduction}
%%%%%%%%%%%%%%%%%%%%%%%%%%%%%%%%%%%
%%%%%%%%%%%%%%%%%%%%%%%%%%%%%%%%%%%

% \rn{This article is a sequel to the article \textit{On toral posets and contact Lie poset algebras} (see \textbf{\cite{contacttoral}}, J. Geom. Phys., 2023).} Lie poset algebras are subalgebras of the classical Lie algebras which lie between a fixed \rn{Cartan} and corresponding \rn{Borel} subalgebra. Recent work has focused on studying the structure of these algebras combinatorially \textbf{\cite{CG, seriesA, Binary, ContactLiePoset, breadth}}\rn{, and the main result of \textbf{\cite{contacttoral}} is the description of a procedure by which one can construct contact, type-A Lie poset algebras.} Of interest here is the work of (\textbf{\cite{BCD}}, 2021), where the authors focus on Lie poset algebras in $\mathfrak{sp}(2n)$ and $\mathfrak{so}(n)$. In particular, the authors of \textbf{\cite{BCD}} associate posets to such ``type-B, C, and D" Lie poset algebras and then identify, among a restricted class, those algebras which are Frobenius in terms of an associated graph. Building on the work in both \textbf{\cite{ContactLiePoset}} and \textbf{\cite{BCD}}, we initiate an investigation into contact Lie poset algebras of types B, C, and D.

This article is a sequel to the articles \textit{Contact Lie poset algebras} (see \textbf{\cite{ContactLiePoset}}, Electron. J. Comb., 2022) and \textit{The index and spectrum of Lie poset algebras of types B, C, and D} (see \textbf{\cite{BCD}}, Electron. J. Comb., 2021). In \textbf{\cite{ContactLiePoset}}, the authors provide a complete characterization of certain contact subalgebras of $\mathfrak{sl}(n).$ The characterized contact Lie algebras were members of the family of ``type-A Lie poset algebras." In \textbf{\cite{BCD}}, the authors, in particular, extend the definition of ``Lie poset algebra" to the other classical types and construct combinatorial ``index" formulas for such algebras. Here, we initiate an investigation into contact Lie poset algebras of types B, C, and D. To establish our main results, it is necessary to extend some key results of \textbf{\cite{BCD}} concerning the index of such algebras.

% In \textbf{\cite{contacttoral}}, the authors provide the description of a procedure by which one can construct certain contact subalgebras of $\mathfrak{sl}(n)$. The resulting contact Lie algebras were members of the family of ``type-A Lie poset algebras". Here, we initiate an investigation into the analogous contact subalgebras of $\mathfrak{sp}(2n)$ and $\mathfrak{so}(n)$, that is, contact Lie poset algebras of types-B, C, and D. To establish our main results, it is necessary to extend some key results of (\textbf{\cite{BCD}}, 2021) concerning the index of such algebras.

Briefly, recall that a $(2k+1)-$dimensional Lie algebra is \textit{contact} if it admits a linear one-form $\varphi$ satisfying $\varphi\wedge(d\varphi)^k\neq 0.$ Such a $\varphi$ is called a (left-invariant) \textit{contact form} and generates a volume form on the algebra's underlying Lie group. The problem of characterizing contact Lie algebras has its roots in the work of Boothby and Wang (\textbf{\cite{BW}}, 1958) and has garnered significant recent attention (see \textbf{\cite{Alvarez,Diatta,GR,Khakimdjanov,RS,InvCon}}, cf. \textbf{\cite{contacttoral}}). Here, we are concerned with identifying contact Lie algebras among Lie poset subalgebras of $\mathfrak{sp}(2n)$ and $\mathfrak{so}(n).$

% Formally, an odd-dimensional Lie algebra $\mathfrak{g}$ is \textit{contact} if 
% there exists $\varphi\in\mathfrak{g}^*$ such that $\varphi  \wedge (d\varphi)^n\ne 0$, where dim $\mathfrak{g}=2n+1$. The one-form
% $\varphi$ is called a \textit{contact form} and  the $(2n+1)$-form $\varphi  \wedge (d\varphi)^n\ne 0$ is a \textit{volume form} on the underlying Lie group.  The construction and classification of contact manifolds is a central problem in differential topology (see \textbf{\cite{wein}}). Here, we are concerned with identifying contact Lie algebras among type-B, C, and D Lie poset algebras.

The authors of \textbf{\cite{BCD}} -- following the suggestion put forth by Coll and Gerstenhaber in (\textbf{\cite{CG}}, 2016) -- define a Lie poset subalgebra of a classical simple Lie algebra $\mathfrak{g}$ as any subalgebra $\mathfrak{p}$ of $\mathfrak{g}$ satisfying $\mathfrak{h}\subset\mathfrak{p}\subset\mathfrak{b},$ where $\mathfrak{h}$ is a Cartan subalgebra of $\mathfrak{g},$ and $\mathfrak{b}$ is a Borel subalgebra of $\mathfrak{g}$ corresponding to $\mathfrak{h}.$ From this definition, it follows that, up to conjugation, each Lie poset algebra can be reckoned as a Lie algebra consisting of upper-triangular matrices whose potentially nonzero entries correspond to relations in an associated poset. In the cases where $\mathfrak{g}$ is $\mathfrak{so}(2n+1),$ $\mathfrak{sp}(2n),$ or $\mathfrak{so}(2n),$ the posets associated with a Lie poset algebra $\mathfrak{p}\subset\mathfrak{g}$ are called type-B, C, or D posets, respectively (see Definition~\ref{def:BCDposet} below).

Now, recall that in \textbf{\cite{ContactLiePoset}}, the method by which contact Lie poset algebras were identified relied heavily upon a Lie-algebraic invariant called the \textit{index}, which is defined as $$\ind\mathfrak{g}=\min_{\varphi\in\mathfrak{g}^*}\dim(\ker(d\varphi)).$$ Each contact Lie algebra necessarily has index one, and this fact is used in the prequel to identify candidate contact algebras among the family of type-A Lie poset algebras, which are then subsequently shown to be contact.
% When attempting to identify contact Lie algebras, a helpful tool is an invariant called the ``index." The \textit{index} of a Lie algebra was first introduced by Dixmier (\textbf{\cite{D}}, 1977) and is defined as $$\ind\mathfrak{g}=\min_{\varphi\in\mathfrak{g}^*}\dim(\ker(d\varphi)).$$ With respect to contact Lie algebras, such algebras must have index one.\footnote{The converse is not true in general. For example, the Lie algebra $\mathfrak{g}=\langle e_1,e_2,e_3\rangle$ with relations $[e_1,e_2]=e_2$ and $[e_1,e_3]=e_3$ has index one but is not contact.} This fact is used in both \textbf{\cite{ContactLiePoset}} and \textbf{\cite{contacttoral}} to identify candidate contact algebras among the family of type-A Lie poset algebras, which are then subsequently shown to be contact. 
In order to initiate a similar study of contact type-B, C, and D Lie poset algebras, we leverage a main result of \textbf{\cite{BCD}}, which is a combinatorial index formula that applies to a restricted class of such algebras associated with posets of ``height one," i.e., whose chains have cardinality at most two. The first main goal of this article is to extend the combinatorial index formula of \textbf{\cite{BCD}} to all type-B, C, and D Lie poset algebras associated with posets of height one (see Section~\ref{sec:indf}). Upon achieving this goal, we then characterize contact Lie poset algebras of types B, C, and D associated with height-one posets.

% Moving to Lie poset algebras of types-B, C, and D, the authors of \textbf{\cite{BCD}} determined a combinatorial index formula which applies to a restricted class of such algebras associated with posets whose chains have cardinality at most two. The authors of \textbf{\cite{BCD}} then use their index formula to identify Frobenius Lie poset algebras of types B, C, and D. One of the first main goals of this paper it to extend the combinatorial index formula found in \textbf{\cite{BCD}} as well as the corresponding characterization of Frobenius algebras so that they apply to all type-B, C, and D Lie poset algebras whose associated posets have chains of cardinality at most two. Having done this, we then characterize contact Lie poset algebras of types B, C, and D associated with posets whose chains have cardinality at most two.

The remainder of this paper is organized as follows. In Section~\ref{sec:prelim} we cover the necessary preliminaries from the theory of posets, including definitions and known results concerning posets of types B, C, and D. In Section~\ref{sec:indf}, we extend the index formula of \textbf{\cite{BCD}} as well as the characterization of Frobenius algebras so that they apply to all type-B, C, and D Lie poset algebras whose associated posets have height one. Following this, in Section~\ref{sec:con}, we determine a characterization of contact, type-B, C, and D Lie poset algebras whose associated posets have height one. Finally, in Section~\ref{sec:ep}, we discuss directions for future research.

\section{Preliminaries}\label{sec:prelim}

In this section, we give the necessary preliminaries from the theory of posets.

Recall that a \textit{finite poset} $(\mathcal{P}, \preceq_{\mathcal{P}})$ consists of a finite set $\mathcal{P}\subset\mathbb{Z}$ together with a binary relation $\preceq_{\mathcal{P}}$ which is reflexive, anti-symmetric, and transitive. We further assume that if $x\preceq_{\mathcal{P}}y$ for $x,y\in\mathcal{P}$, then $x\le y$, where $\le$ denotes the natural ordering on $\mathbb{Z}$. When no confusion will arise, we simply denote a poset $(\mathcal{P}, \preceq_{\mathcal{P}})$ by $\mathcal{P}$, and $\preceq_{\mathcal{P}}$ by $\preceq$. 

Let $x,y\in\mathcal{P}$. If $x\preceq y$ and $x\neq y$, then we call $x\preceq y$ a \textit{strict relation} and write $x\prec y$. Recall that if $x\prec y$ and there exists no $z\in \mathcal{P}$ satisfying $x\prec z\prec y$, then $y$ \textit{covers} $x$ and 
$x\prec y$ is a \textit{covering relation}.  Using this language, the \textit{Hasse diagram} of a poset $\mathcal{P}$ can be reckoned as the graph whose vertices correspond to elements of $\mathcal{P}$ and whose edges correspond to covering relations.

\begin{Ex}
Consider the poset $\mathcal{P}=\{1,2,3,4\}$ with $1\prec2\prec3,4$. The Hasse diagram of $\mathcal{P}$ is illustrated in Figure~\ref{fig:Hasse}.
\begin{figure}[H]
$$\begin{tikzpicture}
	\node (1) at (0, 0) [circle, draw = black, fill = black, inner sep = 0.5mm, label=left:{1}]{};
	\node (2) at (0, 1)[circle, draw = black, fill = black, inner sep = 0.5mm, label=left:{2}] {};
	\node (3) at (-0.5, 2) [circle, draw = black, fill = black, inner sep = 0.5mm, label=left:{3}] {};
	\node (4) at (0.5, 2) [circle, draw = black, fill = black, inner sep = 0.5mm, label=right:{4}] {};
    \draw (1)--(2);
    \draw (2)--(3);
    \draw (2)--(4);
    \addvmargin{1mm}
\end{tikzpicture}$$
\caption{Hasse diagram of $\mathcal{P}$}\label{fig:Hasse}
\end{figure}
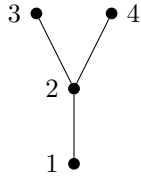
\end{Ex}

\noindent
For a subset $S\subset\mathcal{P}$, the \textit{induced subposet generated by $S$} is the poset $\mathcal{P}_S$ on $S,$ where $i\preceq_{\mathcal{P}_S}j$ if and only if $i,j\in S$ and $i\preceq_{\mathcal{P}}j$. A totally ordered subset $S\subset\mathcal{P}$ is called a \textit{chain}. The \textit{height} of $\mathcal{P}$ is one less than the largest cardinality of a chain in $\mathcal{P}$.

In this article we are interested in a restricted class of posets which generate subalgebras of the classical Lie algebras of types B, C, and D consisting of upper-triangular matrices. In particular, we are interested in the type-B, C, and D posets of \textbf{\cite{BCD}} which are defined as follows.

\begin{definition}\label{def:BCDposet}
A type-C poset is a poset $\mathcal{P}=\{-n,\hdots,-1,1,\hdots, n\}$ such that
\begin{enumerate}
	\item if $i\preceq_{\mathcal{P}}j$, then $i\le j$; and
	\item if $i\neq -j$, then $i\preceq_{\mathcal{P}}j$ if and only if $-j\preceq_{\mathcal{P}}-i$.
\end{enumerate}
A type-D poset is a poset $\mathcal{P}=\{-n,\hdots,-1,1,\hdots, n\}$ satisfying 1 and 2 above as well as 
\begin{enumerate}
    \setcounter{enumi}{2}
    \item $i$ does not cover $-i$, for $i\in \{1,\hdots, n\}$.
\end{enumerate}
A type-B poset is a poset $\mathcal{P}=\{-n,\hdots,-1,0,1,\hdots, n\}$ satisfying 1 through 3 above.  
\end{definition}

\begin{Ex}\label{ex:CHasse}
In Figure~\ref{fig:hasse}, we illustrate the Hasse diagram of the type-C \textup(and D\textup) poset $\mathcal{P}=\{-3,-2,-1,1,2,3\}$ with $-2\prec1,3$; $-3\prec 2$; and $-1\prec 2$. Note that adding 0 to $\mathcal{P}$ and a vertex labeled 0 to the Hasse diagram of Figure~\ref{fig:hasse} results in a type-B poset and its corresponding Hasse diagram.
\begin{figure}[H]
$$\begin{tikzpicture}[scale = 0.65]
	\node (1) at (0, 0) [circle, draw = black, fill=black, inner sep = 0.5mm, label=below:{-3}] {};
	\node (2) at (1,0) [circle, draw = black, fill=black,  inner sep = 0.5mm, label=below:{-2}] {};
	\node (3) at (2, 0) [circle, draw = black, fill=black, inner sep = 0.5mm, label=below:{-1}] {};
    \node (4) at (0, 1) [circle, draw = black, fill=black, inner sep = 0.5mm, label=above:{3}] {};
	\node (5) at (1,1) [circle, draw = black, fill=black,  inner sep = 0.5mm, label=above:{2}] {};
	\node (6) at (2, 1) [circle, draw = black, fill=black, inner sep = 0.5mm, label=above:{1}] {};
    \draw (1)--(5);
    \draw (4)--(2);
    \draw (5)--(3);
    \draw (2)--(6);
\end{tikzpicture}$$
\caption{Hasse diagram of a type-C poset}\label{fig:hasse}
\end{figure}
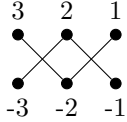
\end{Ex}

Given a type-B, C, or D poset $\mathcal{P}$, let $\mathcal{P}^+=\mathcal{P}_{\mathcal{P}\cap\mathbb{Z}_{>0}}$ and $\mathcal{P}^-=\mathcal{P}_{\mathcal{P}\cap\mathbb{Z}_{<0}}$; that is, $\mathcal{P}^+$ (resp., $\mathcal{P}^-)$ is the poset induced by the positive (resp., negative) elements of $\mathcal{P}$. Let $Rel_{\pm}(\mathcal{P})$ denote the set of relations $x\prec y$ such that $x\in\mathcal{P}^-$ and $y\in\mathcal{P}^+$. We call $\mathcal{P}$ \textit{separable} if $Rel_{\pm}(\mathcal{P})=\emptyset$, and \textit{non-separable} otherwise; note that if $\mathcal{P}$ is a type-B poset which is either separable or of height one, then $0$ cannot be related to any other elements of $\mathcal{P}$. For posets of types B, C, and D, we sometimes use a refined notion of height, saying that such a poset $\mathcal{P}$ is of \textit{height} $(i,j)$ if $\mathcal{P}^+$ (resp., $\mathcal{P})$ is of height $i$ (resp., $j$).

\begin{Ex}
If $\mathcal{P}$ is the poset of Example~\ref{ex:CHasse}, then $\mathcal{P}^+=\{1,2,3\}$ and $\mathcal{P}^-=\{-1,-2,-3\}$; both induced posets have no relations. Further, since $\mathcal{P}$ has chains of cardinality at most two, it follows that $\mathcal{P}$ is of height $(0,1)$.
\end{Ex}

In \textbf{\cite{BCD}} the authors use a condensed version of the Hasse diagram when working with type-B, C, and D posets of height-$(0,1)$, called the ``relation graph". Below we extend the notion of relation graph slightly so that it applies to type-B, C, and D posets of height one.

\begin{definition}\label{def:RG}
Given a type-B, C, or D poset $\mathcal{P}$ of height one, we define the relation graph $RG(\mathcal{P})$ as follows:
\begin{itemize}
    \item each pair of elements $-i,i\in\mathcal{P}$ is represented by a single vertex in $RG(\mathcal{P})$ labeled by $i\in\mathcal{P}^+$ \textup(omitting the vertex representing 0 in type B\textup);
    \item if $-i\prec j$ in $\mathcal{P}$, then there is an edge connecting vertex $i$ and vertex $j$ in $RG(\mathcal{P})$;
    \item if $-i\prec -j$ in $\mathcal{P}$, then there is a dashed edge connecting vertex $i$ and vertex $j$ in $RG(\mathcal{P})$.
\end{itemize}
We denote the vertex set and edge set of $RG(\mathcal{P})$ by $V(\mathcal{P})$ and $E(\mathcal{P})$, respectively. If $RG(\mathcal{P})$ is connected, then $\mathcal{P}$ is called connected.
\end{definition}

\begin{remark}
The extended relations graphs defined above are equivalent to ``signed digraphs", as defined by Reiner \textup(see \textup{\textbf{\cite{Reiner1,Reiner2}}}\textup), with the signs removed and edges with both signs becoming dashed.
\end{remark}

\begin{remark}
If $-i\prec i$ in $\mathcal{P}$, then vertex $i$ defines a self-loop in $RG(\mathcal{P})$. Note that $RG(\mathcal{P})$ can only contain self-loops if $\mathcal{P}$ is a type-C poset.
\end{remark}

\begin{remark}
As in \textup{\textbf{\cite{BCD}}}, we use the notion of connected given in Definition~\ref{def:RG} for type-B, C, and D posets in place of the standard notion in terms of connectedness of the Hasse diagram.
\end{remark}

\begin{Ex}
In Figure~\ref{fig:relationgr}, we illustrate the \textup(a\textup) Hasse diagram and \textup(b\textup) relation graph corresponding to the height-$(1,1)$, type-C poset $\mathcal{P}=\{-3,-2,-1,1,2,3\}$ with $-3\prec -2,1,3$ and $-1,2\prec -3$.
\begin{figure}[H]
$$\begin{tikzpicture}[scale = 0.7]
	\node (-3) at (0, 0) [circle, draw = black, fill=black, inner sep = 0.5mm, label=below:{-2}] {};
	\node (-2) at (1,-1) [circle, draw = black, fill=black,  inner sep = 0.5mm, label=below:{-3}] {};
	\node (-1) at (2, 0) [circle, draw = black, fill=black, inner sep = 0.5mm, label=below:{-1}] {};
    \node (3) at (0, 1) [circle, draw = black, fill=black, inner sep = 0.5mm, label=above:{2}] {};
	\node (2) at (1,2) [circle, draw = black, fill=black,  inner sep = 0.5mm, label=above:{3}] {};
	\node (1) at (2, 1) [circle, draw = black, fill=black, inner sep = 0.5mm, label=above:{1}] {};
	\node (7) at (1, -2.5) {(a)};
	\draw (-3)--(-2)--(1);
    \draw (-2)--(2);
    \draw (3)--(2)--(-1);
\end{tikzpicture}\quad\quad\quad\quad\begin{tikzpicture}[scale = 0.7]
	\node (1) at (0, 0.5) [circle, draw = black,fill=black, inner sep = 0.5mm, label=below:{2}] {};
	\node (2) at (1,0.5) [circle, draw = black,fill=black,  inner sep = 0.5mm, label=below:{3}] {};
	\node (3) at (2, 0.5) [circle, draw = black,fill=black, inner sep = 0.5mm, label=below:{1}] {};
	\node (7) at (1, -2.5) {(b)};
	\draw[dashed] (1)--(2);
    \draw (2)--(3);
	\draw (1,0.5) .. controls (0.5,1.5) and (1.5,1.5) .. (1,0.5);
\end{tikzpicture}$$
\caption{(a) Hasse diagram and (b) relation graph of type-C poset}\label{fig:relationgr}
\end{figure}
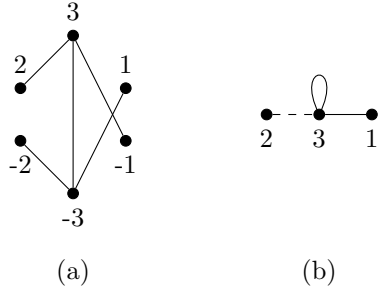
\end{Ex}

One obtains a subalgebra of the appropriate classical Lie algebra from a poset of type B, C, or D as described in the following theorem. As in the introduction, we let $E_{i,j}$ denote an appropriately sized square matrix containing a 1 in the $i,j$-entry and 0's elsewhere; the size of $E_{i,j}$ will be clear from context.

\begin{theorem}[\textbf{\cite{BCD}}]
Type-C \textup(resp., B or D\textup) posets $\mathcal{P}$ are in bijective correspondence with type-C \textup(resp., B or D\textup) Lie poset algebras $\mathfrak{p}$ as follows:
\begin{itemize}
    \item $|\mathcal{P}|=n$ if and only if $\mathfrak{p}\subset \mathfrak{sp}(n)$ \textup(resp., $\mathfrak{p}\subset \mathfrak{so}(n)$\textup);
    \item $-i,i\in \mathcal{P}$ if and only if $E_{-i,-i}-E_{i,i}\in \mathfrak{p}$;
    \item $-i\prec_{\mathcal{P}}-j$ and $j\prec_{\mathcal{P}}i$ if and only if $E_{-i,-j}-E_{j,i}\in\mathfrak{p}$;
    \item $-i\prec_{\mathcal{P}}j$ and $-j\prec_{\mathcal{P}}i$ if and only if $E_{-i,j}+E_{-j,i}\in\mathfrak{p}$ \textup(resp., $E_{-i,j}-E_{-j,i}\in\mathfrak{p}$\textup);
\end{itemize}
and only in type-C
\begin{itemize}
    \item $-i\prec_{\mathcal{P}}i$ if and only if $E_{-i,i}\in\mathfrak{p}$.
\end{itemize}
\end{theorem}

\begin{remark}
Note that as in the type-A case, type-C posets $\mathcal{P}$ determine the matrix form of the corresponding type-C Lie poset algebra by identifying which entries of a $|\mathcal{P}|\times|\mathcal{P}|$ matrix can be non-zero. In particular, the $i,j$-entry can be non-zero if and only if $i\preceq_{\mathcal{P}}j$. The same is almost true in types-B and D, except one ignores relations of the form $-i\prec_{\mathcal{P}} i$.
\end{remark}

\begin{Ex}\label{ex:typeBCD}
Let $\mathcal{P}$ be the poset of Example~\ref{ex:CHasse}. The matrix form encoded by $\mathcal{P}$ and defining the corresponding type-C \textup(and D\textup) Lie poset algebra is illustrated in Figure~\ref{fig:tBCD}, where $*$'s denote potential non-zero entries. 
\begin{figure}[H]
$$\kbordermatrix{
    & -3 & -2 & -1 & 1 & 2 & 3 \\
   -3 & * & 0 & 0 & 0 & * & 0  \\
   -2 & 0 & * & 0 & * & 0 & * \\
   -1 & 0 & 0 & * & 0 & * & 0 \\
   1 & 0 & 0 & 0 & * & 0 & 0 \\
   2 & 0 & 0 & 0 & 0 & * & 0 \\
   3 & 0 & 0 & 0 & 0 & 0 & * \\
  }$$
\caption{Matrix form for $\mathcal{P}=\{-3,-2,-1,1,2,3\}$ with $-2\prec1,3$; $-3\prec 2$; and $-1\prec 2$}\label{fig:tBCD}
\end{figure}
\end{Ex}

Given a type-C poset $\mathcal{P}$, we denote the corresponding type-C Lie poset algebra by $\mathfrak{g}_C(\mathcal{P})$; furthermore, we define the following basis for $\mathfrak{g}_C(\mathcal{P})$:
\begin{align}
\mathscr{B}_C(\mathcal{P})=\{E_{-i,-i}-E_{i,i}~|~-i,i\in\mathcal{P}\}&\cup\{E_{-j,-i}-E_{i,j}~|~-i,-j,i,j\in\mathcal{P},-j\prec -i,i\prec j\} \nonumber \\ 
&\cup\{E_{-i,j}+E_{-j,i}~|~-i,-j,i,j\in\mathcal{P},-j\prec i,-i\prec j\} \nonumber \\
&\cup\{E_{-i,i}~|~-i,i\in\mathcal{P},-i\prec i\}. \nonumber
\end{align}
Similarly, given a type-D \textup(resp., B\textup) poset $\mathcal{P}$ we denote the corresponding type-D \textup(resp., B\textup) Lie poset algebra by $\mathfrak{g}_D(\mathcal{P})$ \textup(resp., $\mathfrak{g}_B(\mathcal{P})$\textup) and define the basis $\mathscr{B}_D(\mathcal{P})$ \textup(resp., $\mathscr{B}_B(\mathcal{P})$\textup) as follows: \begin{align}
\mathscr{B}_D(\mathcal{P})=\{E_{-i,-i}-E_{i,i}~|~-i,i\in\mathcal{P}\}&\cup\{E_{-j,-i}-E_{i,j}~|~-i,-j,i,j\in\mathcal{P},-j\prec -i,i\prec j\} \nonumber \\
&\cup\{E_{-i,j}-E_{-j,i}~|~-i,-j,i,j\in\mathcal{P},-j\prec i,-i\prec j,j<i\}. \nonumber
\end{align}

\noindent
Ongoing, we set the following notational conventions.
\begin{itemize}
    \item $D_i=E_{-i,-i}-E_{i,i}$.
    \item $R^{\pm}_{i,j}=E_{-i,j}+E_{-j,i}$.
    \item $R_{i,j}=E_{-j,-i}-E_{i,j}$ for $i<j$.
    \item $x_{i,n}$ denotes the $i$th standard basis element in $\mathbb{C}^n$, i.e., $$x_{i,n}=[\underbrace{0~\cdots~0}_{i-1}~1~\underbrace{0~\hdots~0}_{n-i}].$$
\end{itemize}

\begin{remark}
Note that $R^{\pm}_{i,j}=R^{\pm}_{j,i}$.
\end{remark}

In \textbf{\cite{BCD}}, the authors establish the following relationships among Lie poset algebras of types-B, C, and D.

\begin{theorem}[Theorems~17 and 18, \textbf{\cite{BCD}}]\label{thm:onlyC}~
\begin{enumerate}
    \item[\textup{(a)}] If $\mathcal{P}$ is a type-D poset such that $-i\nprec i$ for all $i\in\mathcal{P}$, then $\mathfrak{g}_D(\mathcal{P})$ is isomorphic to $\mathfrak{g}_C(\mathcal{P})$.
    \item[\textup{(b)}] If $\mathcal{P}$ is a type-B poset for which 0 is not related to any other element of $\mathcal{P}$, $-i\nprec i$ for all $i\in\mathcal{P}$, and $\mathcal{P}_0=\mathcal{P}_{\mathcal{P}\backslash\{0\}}$, then $\mathfrak{g}_B(\mathcal{P})$ is isomorphic to $\mathfrak{g}_C(\mathcal{P}_0)$.
\end{enumerate}
\end{theorem}

\noindent
With respect to index, the authors of \textbf{\cite{BCD}} established the following index formula for type-B, C, and D Lie poset algebras associated with posets of height $(0,1)$.

\begin{theorem}\label{thm:index01}
Let $\mathcal{P}$ be a height-$(0,1)$ poset of type-B, C, or D and $\mathfrak{g}$ be the corresponding type-B, C, or D Lie poset algebra, respectively. Then $$\ind\mathfrak{g}=|E(\mathcal{P})|-|V(\mathcal{P})|+2\eta(\mathcal{P}),$$ where $\eta(\mathcal{P})$ denotes the number of connected components of $RG(\mathcal{P})$ containing no odd cycles.
\end{theorem}

\noindent
Moreover, combining Theorem 30 of \textbf{\cite{BCD}} with Theorems 2 and 4 of \textbf{\cite{seriesA}} we obtain the following.

\begin{theorem}\label{thm:indsep}
    Let $\mathcal{P}$ be a separable, type-B, C, of D poset of height one and $\mathfrak{g}$ be the corresponding type-B, C, or D Lie poset algebra, respectively. Then $$\ind\mathfrak{g}=|E(\mathcal{P})|-|V(\mathcal{P})|+2.$$
\end{theorem}

In the following section, we combine and extend Theorems~\ref{thm:index01} and~\ref{thm:indsep} to obtain a single index formula which applies to all type-B, C, and D Lie poset algebras associated with posets of height one.

\section{Extended Index Formula}\label{sec:indf}

In this section, for the sake of brevity, all results will be stated for type-C Lie poset algebras; considering Theorem~\ref{thm:onlyC}, though, all results still apply with ``type-C" replaced by ``type-B" or ``type-D". The main result of this section is Theorem~\ref{thm:index1} below.

\begin{theorem}\label{thm:index1}
If $\mathcal{P}$ is a type-C poset of height one and $\mathfrak{g}=\mathfrak{g}_C(\mathcal{P})$, then $$\ind\mathfrak{g}=|E(\mathcal{P})|-|V(\mathcal{P})|+2\eta(\mathcal{P}),$$ where $\eta(\mathcal{P})$ denotes the number of connected components of $RG(\mathcal{P})$ containing no odd cycles.
\end{theorem}

To prove Theorem~\ref{thm:index1}, we make use of an alternative characterization of the index. Let $\mathfrak{g}$ be an $n$-dimensional Lie algebra with ordered basis $\mathscr{B}(\mathfrak{g})=\{E_1,\dots,E_n\}$, and define $$\mathcal{C}(\mathfrak{g},\mathscr{B}(\mathfrak{g}))=([E_i,E_j])_{1\leq i,j\leq n}$$ to be the \textit{commutator matrix} associated with $\mathfrak{g}$. Now, for any $\varphi\in\mathfrak{g}^*$, define the matrix
  $$\varphi\left(\mathcal{C}(\mathfrak{g},\mathscr{B}(\mathfrak{g}))\right)=(\varphi([E_i,E_j]))_{1\leq i,j\leq n}.$$  Using the above notation, we have that $$\ind\mathfrak{g}=\dim\mathfrak{g}-\max_{\varphi\in \mathfrak{g^*}}~\text{rank}~\varphi\left(\mathcal{C}(\mathfrak{g},\mathscr{B}(\mathfrak{g}))\right).$$

Now, in proving Theorem~\ref{thm:index1}, the following result of \textbf{\cite{BCD}} allows us to focus on connected posets.

\begin{theorem}[Theorem~46, \textbf{\cite{BCD}}]\label{thm:disjoint}
If $\mathcal{P}$ is a type-C poset of height one such that $RG(\mathcal{P})$ consists of connected components $\{K_1,\hdots,K_n\}$, then $$\ind\mathfrak{g}_C(\mathcal{P})=\sum_{i=1}^n\ind\mathfrak{g}_C(\mathcal{P}_{K_i}),$$ where $\mathcal{P}_{K_i}$ is the unique type-C poset satisfying $RG(\mathcal{P}_{K_i})=K_i$.
\end{theorem}

\begin{remark}
In \textup{\textbf{\cite{BCD}}}, Theorem~\ref{thm:disjoint} was stated only for type-C Lie poset algebras associated with posets of height $(0,1)$, but the proof applies to type-C Lie poset algebras associated with posets of height one in general.
\end{remark}

Considering Theorems~\ref{thm:index01},~\ref{thm:indsep}, and~\ref{thm:disjoint}, our first step towards proving Theorem~\ref{thm:index1} is to establish an index formula which applies to type-C Lie poset algebras associated with connected, type-C posets of height $(1,1)$. As the desired index formula is in terms of $RG(\mathcal{P})$, let us first determine restrictions on the extended version of $RG(\mathcal{P})$ required for $\mathcal{P}$ to be of height one. In Figure~\ref{fig:subgraphs}, we illustrate all possible subgraphs of $RG(\mathcal{P})$ consisting of two adjacent edges with at least one dashed edge.

\begin{figure}[H]
    \centering
    $$\begin{tikzpicture}[scale=0.65]
        \node (1) at (0,0) [circle, draw = black,fill=black, inner sep = 0.5mm, label=below:{i}] {};
        \node (2) at (1,0) [circle, draw = black,fill=black, inner sep = 0.5mm, label=below:{j}] {};
        \draw (1,0) .. controls (0.5,0.75) and (1.5,0.75) .. (1,0);
        \draw[dashed] (1)--(2);
        \node at (0.5, -2) {(a)};
    \end{tikzpicture}\quad\quad\begin{tikzpicture}[scale=0.65]
    \node (1) at (0,0) [circle, draw = black,fill=black, inner sep = 0.5mm, label=below:{i}] {};
    \node (2) at (1,0) [circle, draw = black,fill=black, inner sep = 0.5mm, label=below:{j}] {};
  \draw[dashed] (0,0) to[bend right=60] (1,0);
  \draw (1,0) to[bend right=60] (0,0);
  \node at (0.5, -2) {(b)};
\end{tikzpicture}\quad\quad\begin{tikzpicture}[scale=0.65]
    \node (1) at (0,0) [circle, draw = black,fill=black, inner sep = 0.5mm, label=below:{i}] {};
    \node (2) at (1,0) [circle, draw = black,fill=black, inner sep = 0.5mm, label=below:{j}] {};
    \node (3) at (2,0) [circle, draw = black,fill=black, inner sep = 0.5mm, label=below:{k}] {};
    \draw[dashed] (1)--(2);
    \draw (2)--(3);
  \node at (1, -2) {(c)};
\end{tikzpicture}\quad\quad\begin{tikzpicture}[scale=0.65]
    \node (1) at (0,0) [circle, draw = black,fill=black, inner sep = 0.5mm, label=below:{i}] {};
    \node (2) at (1,0) [circle, draw = black,fill=black, inner sep = 0.5mm, label=below:{j}] {};
    \node (3) at (2,0) [circle, draw = black,fill=black, inner sep = 0.5mm, label=below:{k}] {};
    \draw[dashed] (1)--(2);
    \draw[dashed] (2)--(3);
  \node at (1, -2) {(d)};
\end{tikzpicture}$$
    \caption{Subgraphs consisting of dashed and non-dashed edge}
    \label{fig:subgraphs}
\end{figure}
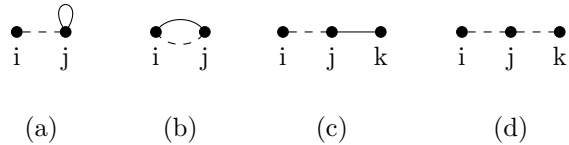

\begin{lemma}\label{lem:nono}
If $\mathcal{P}$ is a type-C poset of height one, then $RG(\mathcal{P})$ does not contain the following subgraphs with vertex set $V$, dashed edge set $E_D$, and non-dashed edge set $E_{\overline{D}}$:
\begin{enumerate}
    \item[\textup{(a)}] $V=\{i,j\}$, $E_D=\{\{i,j\}\}$, $E_{\overline{D}}=\{\{j,j\}\}$ with $i>j$. See Figure~\ref{fig:subgraphs} \textup(a\textup).
    \item[\textup{(b)}] $V=\{i,j\}$, $E_D=\{\{i,j\}\}$, $E_{\overline{D}}=\{\{i,j\}\}$. See Figure~\ref{fig:subgraphs} \textup(b\textup).
    \item[\textup{(c)}] $V=\{i,j,k\}$, $E_D=\{\{i,j\}\}$, $E_{\overline{D}}=\{\{j,k\}\}$ with $i>j$. See Figure~\ref{fig:subgraphs} \textup(c\textup).
    \item[\textup{(d)}] $V=\{i,j,k\}$, $E_D=\{\{i,j\},\{j,k\}\}$ with $i>j>k$ or $i<j<k$. See Figure~\ref{fig:subgraphs} \textup(d\textup).
\end{enumerate}
\end{lemma}
\begin{proof}
Assume otherwise. For (a), the poset $\mathcal{P}$ would contain the chain $-j\prec j\prec i$. For (b), assume without loss of generality that $i>j$. In this case, $\mathcal{P}$ would contain the chain $-i\prec j\prec i$. For (c), the poset would contain the chain $-k\prec j\prec i$. Finally, for (d), assume without loss of generality that $i<j<k$. In this case, $\mathcal{P}$ would contain the chain $i\prec j\prec k$.
\end{proof}

\begin{remark}
    Note that, as a consequence of Lemma~\ref{lem:nono}, the only subgraphs of $RG(\mathcal{P})$ for $\mathcal{P}$ of height one that can occur consisting of two adjacent edges with one dashed edge and one non-dashed are of the form \textup(a\textup) and \textup(c\textup) in Figure~\ref{fig:subgraphs} with $i<j$.
\end{remark}

Next, we show how one can relate the index of a connected, type-C Lie poset algebra associated with a poset of height $(1,1)$ to the index of one associated with a poset of height $(0,1)$. Set \\$\mathcal{C}(\mathfrak{g}_C(\mathcal{P}),\mathscr{B}_C(\mathcal{P}))=\mathcal{C}(\mathfrak{g}_C(\mathcal{P}))$, where the elements of $\mathscr{B}_C(\mathcal{P})$ are ordered as follows: 
\begin{enumerate}
    \item the elements $D_i$ in increasing order of $i$ in $\mathbb{Z}$ followed by
    \item the elements $R_{i,j}$ in increasing lexicographic order of $(i,j)$ in $\mathbb{Z}\times\mathbb{Z}$ followed by
    \item the elements $E_{-i,i}$ in increasing order of $i$ in $\mathbb{Z}$ followed by
    \item the elements $R^{\pm}_{i,j}$ in increasing lexicographic order of $(i,j)$ for $i<j$ in $\mathbb{Z}\times\mathbb{Z}$.
\end{enumerate}

\noindent
With this ordering, since type-C posets of height one have no non-trivial transitivity relations, $\varphi(\mathcal{C}(\mathfrak{g}_C(\mathcal{P})))$ for $\varphi\in\mathfrak{g}^*$ has the form illustrated in Figure~\ref{fig:h01m}.

\begin{figure}[H]
$$\begin{tikzpicture}
  \matrix [matrix of math nodes,left delimiter={(},right delimiter={)}]
  {
    0  & -M_{\varphi}(\mathcal{P})^T   \\       
    M_{\varphi}(\mathcal{P})  & 0   \\       };
\end{tikzpicture}$$
\caption{Matrix form of $\varphi(\mathcal{C}(\mathfrak{g}_C(\mathcal{P})))$, for $\mathcal{P}$ a type-C poset of height one}\label{fig:h01m}
\end{figure}

\noindent
Here, $M_{\varphi}(\mathcal{P})$ is the restriction of $\varphi(\mathcal{C}(\mathfrak{g}_C(\mathcal{P})))$ to rows $i>|V(\mathcal{P})|$ and columns $1\le j\le |V(\mathcal{P})|$, and $-M_{\varphi}(\mathcal{P})^T$ is the restriction of $\varphi(\mathcal{C}(\mathfrak{g}_C(\mathcal{P})))$ to rows $1\le i\le |V(\mathcal{P})|$ and columns $j>|V(\mathcal{P})|$. Since $\text{rank}(M_{\varphi}(\mathcal{P}))=\text{rank}(M_{\varphi}(\mathcal{P})^T)$, to calculate $\ind\mathfrak{g}_C(\mathcal{P})$ it suffices to determine the maximum possible rank of $M_{\varphi}(\mathcal{P})$, for $\varphi\in\mathfrak{g}^*$.

Now, since the $i^{th}$ basis element of $\mathscr{B}_C(\mathcal{P})$ is equal to $D_i$ if $1\le i\le |V(\mathcal{P})|$ and is of the form $R_{j,k}$, $E_{-j,j}$, or $R_{j,k}^{\pm}$ if $i>|V(\mathcal{P})|$, it follows that for each row of $M_{\varphi}(\mathcal{P}),$ there exists a unique element $x\in\mathscr{B}_C(\mathcal{P})$ of the form $R_{j,k},E_{-j,j}$ or $R_{j,k}^{\pm}$ such that all nonzero entries of the row are multiples of $\varphi(x)$. Thus, in determining the maximum possible rank of $M_{\varphi}(\mathcal{P})$, we may assume that $\varphi(x)=1$ for all basis elements $x\in\mathscr{B}_C(\mathcal{P})$; that is, taking $\varphi'\in\mathfrak{g}^*$ satisfying $\varphi'(x)=1$ for all $x\in\mathscr{B}_C(\mathcal{P})$ and setting $M(\mathcal{P})=M_{\varphi'}(\mathcal{P})$, we have
\begin{equation}\label{eqn:index}
\ind\mathfrak{g}_C(\mathcal{P})=\dim\mathfrak{g}_C(\mathcal{P})-2~\text{rank}(M(\mathcal{P})).
\end{equation}

Note that each edge of $RG(\mathcal{P})$ corresponds to a unique row of $M(\mathcal{P})$. Ongoing, it will be helpful to have a generalization of $M(\mathcal{P})$ which can be associated with arbitrary graphs consisting of dashed edges, non-dashed self-loops, and non-dashed edges. Thus, we extend the definition of $M(\mathcal{P})$ as follows.

\begin{definition}
Given a graph $G=(V,E)$, define $M(G)$ to be the $|E|\times |V|$ matrix where for each 
\begin{itemize}
    \item dashed edge $e=\{i,j\}\in E$ with $i<j$, there exists a corresponding row of $M(G)$, denoted $\mathbf{R_{i,j}(G)}$, of the form $x_{i,|V|}-x_{j,|V|}$;
    \item non-dashed self-loop $e=\{i,i\}\in E$, there exists a corresponding row of $M(G)$, denoted $\mathbf{E_{-i,i}(G)}$, of the form $-2x_{i,|V|}$;
    \item non-dashed edge $e=\{i,j\}\in E$ with $i\neq j$, there exists a corresponding row of $M(G)$, denoted $\mathbf{R_{i,j}^{\pm}(G)}$ or $\mathbf{R_{j,i}^{\pm}(G)}$, of the form $-x_{i,|V|}-x_{j,|V|}$.
\end{itemize}
\end{definition}

With this definition, if $G=RG(\mathcal{P})$, then (up to rearranging rows) $M(\mathcal{P})=M(G)$.

\begin{remark}\label{rem:ht01}
Note that if a graph $G$ contains only non-dashed edges and no two edges between the same vertices, then there exists a type-C poset $\mathcal{P}$ of height $(0,1)$ such that $RG(\mathcal{P})=G$. To see this, recall that non-dashed edges correspond to relations of the form $-i\prec j$ and $-j\prec i$. For such collections of relations, no nontrivial transitivity relations can arise and antisymmetry is immediate. On the other hand, considering Lemma~\ref{lem:nono}, graphs with both dashed and non-dashed edges may not correspond to type-C posets of height one.
\end{remark}

\begin{proposition}\label{prop:11t011}
Let $\mathcal{P}$ be a connected, non-separable, type-C poset of height $(1,1)$ for which $RG(\mathcal{P})$ is a tree. Then there exists a connected, type-C poset $\mathcal{P}'$ of height $(0,1)$ such that $|V(\mathcal{P})|=|V(\mathcal{P}')|$, $|E(\mathcal{P})|=|E(\mathcal{P}')|$, $RG(\mathcal{P}')$ is a tree, and $\ind\mathfrak{g}_C(\mathcal{P})=\ind\mathfrak{g}_C(\mathcal{P}')$. 
\end{proposition}
\begin{proof}
Let $|E_D|$ denote the number of dashed edges in $RG(\mathcal{P})$. We define $\mathcal{P}'$ by constructing $RG(\mathcal{P}')$ from $RG(\mathcal{P})$, removing one dashed edge at a time and adding a new non-dashed one. Let $G_0=RG(\mathcal{P})$ and $N=|V(\mathcal{P})|$.
\bigskip

\noindent
\textbf{Step 1}: Since $\mathcal{P}$ is connected, non-separable, and of height $(1,1)$, there must exist a dashed edge $e_d$ that shares a vertex with a non-dashed edge $e_{\overline{d}}$. Considering Lemma~\ref{lem:nono}, since $RG(\mathcal{P})$ is a tree, $e_d=\{i,j\}$ with $i<j$ and $e_{\overline{d}}=\{j,k\}$. Form $G_1$ by removing $e_d$ from $G_0$ and adding the non-dashed edge $\{i,k\}$.
Note that $\{i,k\}$ is not an edge of $G_0$, since otherwise $G_0$ would contain a cycle defined by the vertices $i,j,k$, contradicting our assumption that $G_0$ is a tree. Moreover, note that $G_1$ is a tree. To see this, note that if $G_1$ contains a cycle, say $\mathcal{C}$, then by construction $\mathcal{C}$ must contain the edge $\{i,k\}$. There are two cases.
\begin{itemize}
    \item If $\mathcal{C}$ also contains the edge $e_{\overline{d}}$, then replacing $\{i,k\}$ and $e_{\overline{d}}=\{j,k\}$ by $\{i,j\}$ in $\mathcal{C}$ results in a cycle contained in $G_0$, a contradiction.
    \item If $\mathcal{C}$ does not contain the edge $e_{\overline{d}}$, then replacing $\{i,k\}$ by $e_d$ and $e_{\overline{d}}$ in $\mathcal{C}$ results in a cycle of $G_0$, a contradition.
\end{itemize}
Now, since one can form $M(G_1)$ from $M(G_0)$ by replacing $$\mathbf{R_{i,j}(G_0)}=x_{i,N}-x_{j,N}$$ with
\begin{align*}
    -\mathbf{R_{i,j}(G_0)}+\mathbf{R^{\pm}_{j,k}(G_0)}&=-(x_{i,N}-x_{j,N})+(-x_{j,N}-x_{k,N}) \\
    &=-x_{i,N}-x_{k,N}\\
    &=\mathbf{R^{\pm}_{i,k}(G_1)},
\end{align*}
it follows that $\text{rank}(M(G_1))=\text{rank}(M(G_0))=\text{rank}(M(\mathcal{P}))$. Note that, considering Remark~\ref{rem:ht01}, $G_1$ may not correspond to the relations graph of a type-C poset of height one.
\bigskip

\noindent
\textbf{Step m}: By construction, $G_{m-1}$ is a tree with $|V(\mathcal{P})|$ vertices and $|E(\mathcal{P})|$ edges. If there are no dashed edges, we are finished. Otherwise, there exists a dashed edge $e_d=\{i,j\}$ which shares a vertex with a non-dashed edge $e_{\overline{d}}=\{j,k\}$. Note that it is not necessarily the case that $i<j$ because $G_{m-1}$ may not be the relations graph of a type-C poset of height one. Form $G_m$ by replacing $e_d$ in $G_{m-1}$ by the non-dashed edge $\{i,k\}$. Arguing as in Step 1, we find that $\{i,k\}$ is not an edge of $G_{m-1}$ and $G_m$ is a tree. Now, let $\mathbf{R}(G_{m-1})$ denote the row of $M(G_{m-1})$ corresponding to the edge $\{i,j\}$, i.e., $\mathbf{R}(G_{m-1})=\mathbf{R}_{i,j}(G_{m-1})$ if $i<j$ and $\mathbf{R}(G_{m-1})=\mathbf{R}_{j,i}(G_{m-1})$ otherwise. Then we can form $M(G_{m})$ from $M(G_{m-1})$ by replacing $$\mathbf{R(G_{m-1})}=(1-2\delta_{i>j})(x_{i,N}-x_{j,N})$$ of $M(G_{m-1})$ with
\begin{align*}
    (1-2\delta_{i<j})\mathbf{R(G_{m-1})}+\mathbf{R^{\pm}_{j,k}(G_{m-1})}&=(1-2\delta_{i<j})(1-2\delta_{i>j})(x_{i,N}-x_{j,N})+(-x_{j,N}-x_{k,N}) \\
    &=-(x_{i,N}-x_{j,N})+(-x_{j,N}-x_{k,N}) \\
    &=-x_{i,N}-x_{k,N}\\
    &=\mathbf{R^{\pm}_{i,k}(G_{m})}.
\end{align*}
It follows that $\text{rank}(G_{m-1})=\text{rank}(G_{m})$.
\bigskip

\noindent
Since in each step a dashed edge is removed and none are added, it follows that $G_{|E_D|}$ is a tree which contains no dashed edges. Consequently, considering Remark~\ref{rem:ht01}, $G_{|E_D|}$ is the relations graph for some connected, type-C poset $\mathcal{P}'$ of height $(0,1)$. Moreover, our work above shows that $$\text{rank}(M(\mathcal{P}'))=\text{rank}(M(G_{|E_D|}))=\text{rank}(M(\mathcal{P}))$$ and $$\dim\mathfrak{g}_C(\mathcal{P}')=|E(\mathcal{P}')|+|V(\mathcal{P}')|=|E(\mathcal{P})|+|V(\mathcal{P})|=\dim\mathfrak{g}_C(\mathcal{P}).$$ Therefore, $\ind\mathfrak{g}_C(\mathcal{P})=\ind\mathfrak{g}_C(\mathcal{P}')$.
\end{proof}

In order to obtain an analogous result in the case where $RG(\mathcal{P})$ contains a self loop, we require the following technical lemma.

\begin{lemma}\label{lem:path}
If $1\le i_0\neq i_1\neq \hdots\neq i_n\le N$ with $n\ge 1$, then $$\sum_{j=0}^{n-1}(-1)^{j+1}(-x_{i_j,N}-x_{i_{j+1},N})=x_{i_0,N}-(-1)^nx_{i_n,N}.$$
% \begin{enumerate}
%     \item[\textup{(a)}] $$\sum_{j=0}^{n-1}(-1)^{j+1}(-x_{i_j,N}-x_{i_{j+1},N})=x_{i_0,N}-(-1)^nx_{i_n,N}$$ and
%     \item[\textup{(b)}] \rn{(I think you can remove this.)} $$\sum_{j=0}^{n-1}(v_{i_j,N}-v_{i_{j+1},N})=v_{i_0,N}-v_{i_n,N}.$$
% \end{enumerate}
\end{lemma}
\begin{proof}
By induction. When $n=1$ the result is trivial. Assume that the result holds for $n-1\ge 0$. We have
\begin{align*}
    \sum_{j=0}^{n-1}(-1)^{j+1}(-x_{i_j,N}-x_{i_{j+1},N})&=(-1)^n(-x_{i_{n-1},N}-x_{i_n,N})+\sum_{j=0}^{n-2}(-1)^{j+1}(-x_{i_j,N}-x_{i_{j+1},N}) \\
    &=(-1)^n(-x_{i_{n-1},N}-x_{i_n,N})+x_{i_0,N}-(-1)^{n-1}x_{i_{n-1},N}\\
    &=x_{i_0,N}-(-1)^nx_{i_n,N},
\end{align*}
where the second equality follows from our inductive hypothesis. Thus, the result follows by induction. %Finally, (b) is clear.
\end{proof}

\begin{proposition}\label{prop:sloopt}
Let $\mathcal{P}$ be a connected, non-separable, type-C poset of height $(1,1)$ for which $RG(\mathcal{P})$ contains a self-loop. Then there exists a connected, type-C poset $\mathcal{P}'$ of height $(0,1)$ such that $|V(\mathcal{P})|=|V(\mathcal{P}')|$, $|E(\mathcal{P})|=|E(\mathcal{P}')|$, $RG(\mathcal{P}')$ contains a self-loop, and $\ind\mathfrak{g}_C(\mathcal{P})=\ind\mathfrak{g}_C(\mathcal{P}')$. 
\end{proposition}
\begin{proof}
Let $|E_D|$ denote the number of dashed edges in $RG(\mathcal{P})$. As in Proposition~\ref{prop:11t011}, we define $\mathcal{P}'$ by constructing $RG(\mathcal{P}')$ from $RG(\mathcal{P})$. Let $G_0=RG(\mathcal{P})$, $N=|V(\mathcal{P})|$, and $p$ denote a vertex of $G_0$ which defines a self-loop.
\bigskip

\noindent
\textbf{Step $i$}: If $G_{i-1}$ contains no dashed edges, then we are done. Otherwise, since $G_{i-1}$ is a connected graph with a self-loop at vertex $p$, there must exist a path defined by the sequence of vertices  $p=p_0,p_1,\hdots,p_t$ such that $\{p_{t-1},p_t\}$ is a dashed edge and $\{p_{l-1},p_{l}\}$ is non-dashed, for $1\le l<t$. Note that it is not necessarily the case that $p_{t-1}>p_t$ since $G_{i-1}$ may not be the relations graph of a type-C poset of height one. Form $G_i$ by replacing the dashed edge $\{p_{t-1},p_t\}$ by a non-dashed edge between the same vertices. Let $\mathbf{R}(G_{i-1})$ denote the row of $M(G_{i-1})$ corresponding to the edge $\{p_{t-1},p_t\}$, i.e., $\mathbf{R(G_{i-1})}=\mathbf{R_{p_{t},p_{t-1}}(G_{i-1})}$ if $p_{t}<p_{t-1}$ and $\mathbf{R(G_{i-1})}=\mathbf{R_{p_{t-1},p_{t}}(G_{i-1})}$ otherwise. Then we can form $M(G_i)$ from $M(G_{i-1})$ by replacing $$\mathbf{R(G_{i-1})}=(1-2\delta_{p_t>p_{t-1}})(x_{p_t,N}-x_{p_{t-1},N})$$ of $M(G_{i-1})$ with
\begin{align*}
    (1-2\delta_{p_t<p_{t-1}})\mathbf{R(G_{i-1})}&+(-1)^{t-1}\mathbf{E_{-p_0,p_0}(G_{i-1})}+2\sum_{i=2}^t(-1)^i\mathbf{R^{\pm}_{p_{t-i},p_{t-i+1}}(G_{i-1})}\\
    &=(x_{p_{t-1},N}-x_{p_t,N})+(-1)^{t-1}(-2x_{p_0,N})+2\sum_{i=2}^t(-1)^i(-x_{p_{t-i},N}-x_{p_{t-i+1},N}) \\
    &=(x_{p_{t-1},N}-x_{p_t,N})+(-1)^{t-1}(-2x_{p_0,N})+2\sum_{j=0}^{t-2}(-1)^{j+t}(-x_{p_{j},N}-x_{p_{j+1},N}) \\
    &=(x_{p_{t-1},N}-x_{p_t,N})+(-1)^{t-1}(-2x_{p_0,N})+2(-1)^{t-1}\sum_{j=0}^{t-2}(-1)^{j+1}(-x_{p_{j},N}-x_{p_{j+1},N}) \\
    &=(x_{p_{t-1},N}-x_{p_t,N})+(-1)^{t-1}(-2x_{p_0,N})+2(-1)^{t-1}(x_{p_0,N}-(-1)^{t-1}x_{p_{t-1},N}) \\
    &=x_{p_{t-1},N}-x_{p_t,N}+2(-1)^tx_{p_0,N}+2(-1)^{t-1}x_{p_0,N}-2x_{p_{t-1},N} \\
    &=-x_{p_{t-1},N}-x_{p_{t},N}\\
    &=\mathbf{R^{\pm}_{p_t,p_{t-1}}(G_{i})},
\end{align*}
where for the fourth equality we applied Lemma~\ref{lem:path}. It follows that $\text{rank}(M(G_{i-1}))=\text{rank}(M(G_{i}))$.
\bigskip

After Step $|E_D|$, each dashed edge of $RG(\mathcal{P})$ has been replaced by a non-dashed edge. Note that, considering Lemma~\ref{lem:nono}, no pair of vertices in $RG(\mathcal{P})$ can be connected by both a dashed and a non-dashed edge. Consequently, considering Remark~\ref{rem:ht01}, $G_{|E_D|}$ corresponds to $RG(\mathcal{P}')$ for some height$-(0,1)$ poset $\mathcal{P}'$ with a self-loop at vertex $p$. Moreover, our work above shows that $$\text{rank}(RG(\mathcal{P}'))=\text{rank}(M(G_{|E_D|}))=\text{rank}(M(G_0))=\text{rank}(RG(\mathcal{P}))$$ and $$\dim\mathfrak{g}_C(\mathcal{P}')=|E(\mathcal{P}')|+|V(\mathcal{P}')|=|E(\mathcal{P})|+|V(\mathcal{P})|=\dim\mathfrak{g}_C(\mathcal{P}).$$ Therefore, $\ind\mathfrak{g}_C(\mathcal{P})=\ind\mathfrak{g}_C(\mathcal{P}')$.
\end{proof}

To handle the cases where $RG(\mathcal{P})$ contains an even or an odd cycle (consisting of more than one edge), we require the following two lemmas.

\begin{lemma}\label{lem:oddcl2}
Let $1\le i_0\neq i_1\neq \cdots\neq i_n\le N$, $L_j=x_{i_j,N}-x_{i_{j+1},N}$ or $-x_{i_j,N}-x_{i_{j+1},N}$ for $0\le j<n$, and $L_n=x_{i_n,N}-x_{i_0,N}$ or $-x_{i_n,N}-x_{i_0,N}$. If $$|\{L_k~|~L_k=-x_{i_j,N}-x_{i_{j+1},N}\text{ or }-x_{i_n,N}-x_{i_0,N}\}|$$ is odd, then there exist constants $c_j\in\{-1,1\}$ such that $\sum_{j=0}^{n}c_jL_j=-2x_{i_0,N}$.
\end{lemma}
\begin{proof}
By induction on $n$. If $n=1$, the result is trivial. Assume the result holds for $n-1\ge 0$. There are three cases.
\bigskip

\noindent
\textbf{Case 1:} $L_j=-x_{i_j,N}-x_{i_{j+1},N}$ for $0\le j\le n-1$ and $L_n=-x_{i_n,N}-x_{i_0,N}$. In this case, note that $n$ must be even. Consequently,
\begin{align*}
    L_n+\sum_{j=0}^{n-1}(-1)^jL_j&=(-x_{i_n,N}-x_{i_0,N})+\sum_{j=0}^{n-1}(-1)^j(-x_{i_j,N}-x_{i_{j+1},N}) \\
    &=(-x_{i_n,N}-x_{i_0,N})-\sum_{j=0}^{n-1}(-1)^{j+1}(-x_{i_j,N}-x_{i_{j+1},N}) \\
    &=(-x_{i_n,N}-x_{i_0,N})-(x_{i_0,N}-(-1)^nx_{i_n,N}) \\
    &=-2x_{i_0,N},
\end{align*}
where the third equality follows from Lemma~\ref{lem:path}.
\bigskip

\noindent
\textbf{Case 2:} There exists $0\le k< n-1$ such that $L_k=x_{i_k,N}-x_{i_{k+1},N}$ and $L_{k+1}=-x_{i_{k+1},N}-x_{i_{k+2},N}$, or $L_k=-x_{i_k,N}-x_{i_{k+1},N}$ and $L_{k+1}=x_{i_{k+1},N}-x_{i_{k+2},N}$. Assume that $L_k=x_{i_k,N}-x_{i_{k+1},N}$ and $L_{k+1}=-x_{i_{k+1},N}-x_{i_{k+2},N}$; the other case follows via similar reasoning (replacing subtraction by addition). Note that, in this case, $L_{k+1}+(-1)L_k=-x_{i_k,N}-x_{i_{k+2},N}$. Consequently, applying the induction hypothesis to the sequence of vectors $$L'_j=\begin{cases}
    L_j, & 0\le j<k \\
    L_{k+1}+(-1)L_k, & j=k \\
    L_{j+1}, & k<j\le n-1
\end{cases},$$
for $0\le j\le n-1$, the result follows.
\bigskip

\noindent
\textbf{Case 3:} $L_{n-1}=x_{i_{n-1},N}-x_{i_n,N}$ and $L_n=-x_{i_n,N}-x_{i_0,N}$ or $L_{n-1}=-x_{i_{n-1},N}-x_{i_n,N}$ and $L_n=x_{i_n,N}-x_{i_0,N}$. Assume that $L_{n-1}=x_{i_{n-1},N}-x_{i_n,N}$ and $L_n=-x_{i_n,N}-x_{i_0,N}$; the other case follows via similar reasoning (replacing subtraction by addition). Note that in this case $L_{n}+(-1)L_{n-1}=-x_{i_{n-1},N}-x_{i_0,N}$. Consequently, applying the induction hypothesis to the sequence of vectors $$L'_j=\begin{cases}
    L_j, & 0\le j<n-1 \\
    L_{n}+(-1)L_{n-1}, & j=n-1
\end{cases},$$
for $0\le j\le n-1$, the result follows.
\end{proof}

\begin{lemma}\label{lem:evencl}
Let $1\le i_0\neq i_1\neq\cdots\neq i_n\le N$, $L_j=x_{i_j,N}-x_{i_{j+1},N}$ or $-x_{i_j,N}-x_{i_{j+1},N}$ for $0\le j<n$, and $L_n=x_{i_n,N}-x_{i_0,N}$ or $-x_{i_n,N}-x_{i_0,N}$. If $$|\{L_k~|~L_k=-x_{i_j,N}-x_{i_{j+1},N}\text{ or }-x_{i_n,N}-x_{i_0,N}\}|$$ is even, then there exists constants $c_j\in\{-1,1\}$ such that $\sum_{j=0}^{n}c_jL_j=0$.
\end{lemma}
\begin{proof}
By induction on $n$. If $n=1$, the result is trivial. Assume the result holds for $n-1\ge 0$. There are four cases.
\bigskip

\noindent
\textbf{Case 1:} $L_j=-x_{i_j,N}-x_{i_{j+1},N}$ for $0\le j\le n-1$ and $L_n=-x_{i_n,N}-x_{i_0,n}$. In this case, note that $n$ is odd. Consequently,
\begin{align*}
    -L_n-\sum_{j=0}^{n-1}(-1)^{j+1}L_j&=-(-x_{i_n,N}-x_{i_0,N})-\sum_{j=0}^{n-1}(-1)^{j+1}(-x_{i_j,N}-x_{i_{j+1},N}) \\
    &=(x_{i_n,N}+x_{i_0,N})-(x_{i_0,N}-(-1)^nx_{i_n,N}) \\
    &=x_{i_n,N}+x_{i_0,N}-(x_{i_0,N}+x_{i_n,N}) \\
    &=0,
\end{align*}
where the second equality follows from Lemma~\ref{lem:path}.
\bigskip

\noindent
\textbf{Case 2:} $L_j=x_{i_j,N}-x_{i_{j+1},N}$ for $0\le j\le n-1$ and $L_n=x_{i_n,N}-x_{i_0,N}$. In this case, we find that
\begin{align*}
    L_n+\sum_{j=0}^{n-1}L_j&=(x_{i_n,N}-x_{i_0,N})+\sum_{j=0}^{n-1}(x_{i_j,N}-x_{i_{j+1},N}) \\
    &=(x_{i_n,N}-x_{i_0,N})+(x_{i_0,N}-x_{i_n,N}) \\
    &=0.
\end{align*}
\bigskip

\noindent
\textbf{Case 3:} There exists $0\le k< n-1$ such that $L_k=x_{i_k,N}-x_{i_{k+1},N}$ and $L_{k+1}=-x_{i_{k+1},N}-x_{i_{k+2},N}$, or $L_k=-x_{i_k,N}-x_{i_{k+1},N}$ and $L_{k+1}=x_{i_{k+1},N}-x_{i_{k+2},N}$. Assume that $L_k=x_{i_k,N}-x_{i_{k+1},N}$ and $L_{k+1}=-x_{i_{k+1},N}-x_{i_{k+2},N}$; the other case follows via similar reasoning (replacing subtraction by addition). Note that, in this case, $L_{k+1}+(-1)L_k=-x_{i_k,N}-x_{i_{k+2},N}$. Consequently, applying the induction hypothesis to the sequence of vectors 
$$L'_j=\begin{cases}
    L_j, & 0\le j<k \\
    L_{k+1}+(-1)L_k, & j=k\\
    L_{j+1}, & k<j\le n-1
\end{cases}$$
for $0\le j\le n-1$, the result follows.
\bigskip

\noindent
\textbf{Case 4:} $L_{n-1}=x_{i_{n-1},N}-x_{i_n,N}$ and $L_n=-x_{i_n,N}-x_{i_0,N}$ or $L_{n-1}=-x_{i_{n-1},N}-x_{i_n,N}$ and $L_n=x_{i_n,N}-x_{i_0,N}$. Assume that $L_{n-1}=x_{i_{n-1},N}-x_{i_n,N}$ and $L_n=-x_{i_n,N}-x_{i_0,N}$; the other case follows via similar reasoning (replacing addition by subtraction). Note that, in this case, $L_n+(-1)L_{n-1}=-x_{i_{n-1},N}-x_{i_0,N}$. Consequently, applying the induction hypothesis to the sequence of vectors 
$$L'_j=\begin{cases}
    L_j, & 0\le j<n-1 \\
    L_n+(-1)L_{n-1}, & j=n-1
\end{cases}$$
for $0\le j\le n-1$, the result follows.
\end{proof}

\begin{proposition}\label{prop:ocyclet}
Let $\mathcal{P}$ be a connected, non-separable, type-C poset of height $(1,1)$ for which $RG(\mathcal{P})$ contains an odd cycle and no self-loops. Then there exists a connected, type-C poset $\mathcal{P}'$ of height $(0,1)$ such that $|V(\mathcal{P})|=|V(\mathcal{P}')|$, $|E(\mathcal{P})|=|E(\mathcal{P}')|$, $RG(\mathcal{P}')$ contains a self-loop, and $\ind\mathfrak{g}_C(\mathcal{P})=\ind\mathfrak{g}_C(\mathcal{P}')$. 
\end{proposition}
\begin{proof}
Set $N=|V(\mathcal{P})|$ and $G=RG(\mathcal{P})$. Let $\mathcal{C}$ denote an odd cycle of $G$ and assume that $\mathcal{C}$ is defined by the sequence of vertices $p_0,p_1,\hdots,p_n,p_0$. Set 
\begin{align*}
L_j&=\begin{cases}\mathbf{R^{\pm}_{p_j,p_{j+1}}(G)}, & \{p_j,p_{j+1}\}\text{ is non-dashed in }G \\
\mathbf{R_{p_j,p_{j+1}}(G)}, & \{p_j,p_{j+1}\}\text{ is dashed in }G\text{ and }p_j<p_{j+1}\\
-\mathbf{R_{p_{j+1},p_{j}}(G)}, & \{p_j,p_{j+1}\}\text{ is dashed in }G\text{ and }p_{j+1}<p_{j} \\
\end{cases} \\
&=\begin{cases}
-x_{p_j,N}-x_{p_{j+1},N}, & \{p_j,p_{j+1}\}\text{ is non-dashed in }G \\
x_{p_j,N}-x_{p_{j+1},N}, & \{p_j,p_{j+1}\}\text{ is dashed in }G \\
\end{cases}
\end{align*}
for $0\le j\le n-1$ and
\begin{align*}
L_n&=\begin{cases}\mathbf{R^{\pm}_{p_n,p_0}(G)}, & \{p_n,p_0\}\text{ is non-dashed in }G \\
\mathbf{R_{p_n,p_0}(G)}, & \{p_n,p_0\}\text{ is dashed in }G\text{ and }p_n<p_0 \\
-\mathbf{R_{p_0,p_n}(G)}, & \{p_n,p_0\}\text{ is dashed in }G\text{ and }p_0<p_n
\end{cases} \\
&=\begin{cases}
-x_{p_n,N}-x_{p_0,N}, & \{p_0,p_n\}\text{ is non-dashed in }G \\
x_{p_n,N}-x_{p_0,N}, & \{p_0,p_n\}\text{ is dashed in }G. \\
\end{cases}
\end{align*}
Note that by Lemma~\ref{lem:nono} (d), if $\{p_s,p_r\}$ and $\{p_r,p_t\}$ are adjacent dashed edges of $C,$ then either $p_r<p_s,p_t$ or $p_r>p_s,p_t.$ It then follows that $\mathcal{C}$ must contain at least one non-dashed edge since $\mathcal{C}$ is defined by an odd number of vertices. Further, by Lemma~\ref{lem:nono} (c), if $\{p_s,p_r\}$ is a dashed edge of $\mathcal{C}$ and $\{p_r,p_t\}$ is a non-dashed edge of $C,$ then $p_r<p_s.$ Consequently, each path in $\mathcal{C}$ that consists entirely of dashed edges and is maximal under containment contains an odd number of vertices, i.e., an even number of edges. Since $\mathcal{C}$ is defined by an odd number of edges, it follows that $|\{L_j~|~L_j=-x_{p_j,N}-x_{p_{j+1},N}\text{ or }-x_{p_n,N}-x_{p_0,N}\}|$ is odd, and Lemma~\ref{lem:oddcl2} implies that there exist constants $c_j\in\{-1,1\}$ such that $$\sum_{j=0}^nc_jL_j=-2x_{p_0,N}.$$ Denote by $G'$ the graph formed from $G$ by removing the dashed edge $\{p_0,p_n\}$ and adding a self-loop at vertex $p_0$. Let $\mathbf{R(G)}$ denote the row of $M(G)$ corresponding to $\{p_0,p_n\}$, i.e., $\mathbf{R(G)}=\mathbf{R_{p_0,p_n}(G)}$ if $p_0<p_n$ and $\mathbf{R(G)}=\mathbf{R_{p_n,p_0}(G)}$ otherwise. Since one can form $M(G')$ from $M(G)$ by replacing $\mathbf{R(G)}$ with $$\sum_{j=0}^nc_jL_j=-2x_{p_0,N}=\mathbf{E_{-p_0,p_0}(G')},$$ it follows that $\text{rank}(M(G'))=\text{rank}(M(G))$. Now, applying a recursive argument similar to that in the proof of Proposition~\ref{prop:sloopt}, the result follows.
\end{proof}

% \bl{Consequently, $(\ast)$ the sequence of vertices defining either a cycle or path consisting of dashed edges in $C$ must alternate between being the smallest and largest vertices of the containing edges. Since $C$ is defined by an odd number of vertices, applying $(\ast)$ it follows that $C$ must contain a non-dashed edge. Now note that by Lemma~\ref{lem:nono} (c), adjacent dashed and non-dashed edges of $C$ must share the larger vertex of the dashed edge. So, applying $(\ast)$ once again, we find that paths of $C$ consisting of dashed edges and maximal under containment must consist of an even number of edges. Consequently, it follows that an odd number of the $L_j$ must be of the form $-x_{p_j,N}-x_{p_{j+1},N}$ or $-x_{p_n,N}-x_{p_0,N}$.} Thus, applying Lemma~\ref{lem:oddcl2}, there exist constants $c_j\in\{-1,1\}$ such that $$\sum_{j=0}^nc_jL_j=-2x_{p_0,N}.$$ 

\begin{proposition}\label{prop:evencl}
Let $\mathcal{P}$ be a connected, non-separable, type-C poset of height $(1,1)$ for which $RG(\mathcal{P})$ contains an even cycle and no odd cycles. Then there exists a connected, type-C poset $\mathcal{P}'$ of height $(0,1)$ such that $RG(\mathcal{P}')$ is a tree, $|V(\mathcal{P})|=|V(\mathcal{P}')|$, and $\ind\mathfrak{g}_C(\mathcal{P})=\ind\mathfrak{g}_C(\mathcal{P}')+|E(\mathcal{P})|-|V(\mathcal{P})|+1$. 
\end{proposition}
\begin{proof}
Set $G=RG(\mathcal{P})$ and $N=|V(\mathcal{P})|$. Let $\mathcal{C}$ denote an even cycle in $G$ defined by the sequence of vertices $p_0,p_1,\hdots,p_n,p_0$. Since $\mathcal{P}$ is non-separable, there must exist a non-dashed edge $e$ in $G$. Assume that $e\neq \{p_0,p_n\}$. Note that this does not imply that $\{p_0,p_n\}$ is a dashed edge. Let $G'$ denote the graph formed from $G$ by removing the edge $\{p_0,p_n\}$. We claim that $\text{rank}(M(G))=\text{rank}(M(G'))$. To see this, set 
\begin{align*}
L_j&=\begin{cases}\mathbf{R^{\pm}_{p_j,p_{j+1}}(G)}, & \{p_j,p_{j+1}\}\text{ is non-dashed in }G \\
\mathbf{R_{p_j,p_{j+1}}(G)}, & \{p_j,p_{j+1}\}\text{ is dashed in }G\text{ and }p_j<p_{j+1}\\
-\mathbf{R_{p_{j+1},p_{j}}(G)}, & \{p_j,p_{j+1}\}\text{ is dashed in }G\text{ and }p_{j+1}<p_{j} \\
\end{cases} \\
&=\begin{cases}
-x_{p_j,N}-x_{p_{j+1},N}, & \{p_j,p_{j+1}\}\text{ is non-dashed in }G; \\
x_{p_j,N}-x_{p_{j+1},N}, & \{p_j,p_{j+1}\}\text{ is dashed in }G, \\
\end{cases}
\end{align*}
for $0\le j\le n-1$, and
\begin{align*}
L_n&=\begin{cases}\mathbf{R^{\pm}_{p_n,p_0}(G)}, & \{p_n,p_0\}\text{ is non-dashed in }G \\
\mathbf{R_{p_n,p_0}(G)}, & \{p_n,p_0\}\text{ is dashed in }G\text{ and }p_n<p_0 \\
-\mathbf{R_{p_0,p_n}(G)}, & \{p_n,p_0\}\text{ is dashed in }G\text{ and }p_0<p_n
\end{cases} \\
&=\begin{cases}
-x_{p_n,N}-x_{p_0,N}, & \{p_0,p_n\}\text{ is non-dashed in }G; \\
x_{p_n,N}-x_{p_0,N}, & \{p_0,p_n\}\text{ is dashed in }G; \\
\end{cases}
\end{align*}
Arguing as in the proof of Proposition~\ref{prop:ocyclet}, we invoke Lemma~\ref{lem:nono} (c) and (d) to find that $\mathcal{C}$ must contain an even number of non-dashed edges, i.e., $|\{L_j~|~L_j=-x_{p_j,N}-x_{p_{j+1},N}\text{ or }-x_{p_n,N}-x_{p_0,N}\}|$ is even. Thus, applying Lemma~\ref{lem:evencl}, there exist constants $c_j\in\{-1,1\}$ such that $$\sum_{j=0}^nc_jL_j=0.$$ Let $\mathbf{R(G)}$ denote the row of $M(G)$ corresponding to $\{p_0,p_n\}$, i.e., $\mathbf{R(G)}=\mathbf{R_{p_0,p_n}(G)}$ if $p_0<p_n$ and $\mathbf{R(G)}=\mathbf{R_{p_n,p_0}(G)}$ otherwise. Since one can form $M(G')$ with an additional zero row from $M(G)$ by replacing $\mathbf{R(G)}$ with $$\sum_{j=0}^nc_jL_j=0,$$ it follows that $\text{rank}(M(G'))=\text{rank}(M(G))$, as claimed.

Now, since $G$ is a finite graph with only even cycles, it is possible to form a tree $G''$ from $G$ by recursively removing edges from even cycles. Moreover, considering our construction of $G'$ from $G$ given above, one can do so in such a way that $G''$ contains a non-dashed edge. Since $G$ has $|E(\mathcal{P})|$ edges and $|V(\mathcal{P})|$ vertices, it follows that one must remove $|E(\mathcal{P})|-|V(\mathcal{P})|+1$ edges to form $G''$. Considering our work above, it follows that $\text{rank}(M(\mathcal{P}))=\text{rank}(M(G''))$. Now, applying an argument similar to that of the proof of Proposition~\ref{prop:11t011}, we find that there exists a connected, type-C poset $\mathcal{P}'$ of height $(0,1)$ such that $RG(\mathcal{P}')$ is a tree, $|V(\mathcal{P}')|=|V(G'')|=|V(\mathcal{P})|$, $|E(\mathcal{P}')|=|E(G'')|$, and $\text{rank}(M(\mathcal{P}'))=\text{rank}(M(G''))=\text{rank}(M(\mathcal{P}))$. Thus, we have
\begin{align*}
    \ind\mathfrak{g}_C(\mathcal{P})&=|E(\mathcal{P})|+|V(\mathcal{P})|-2~\text{rank}(M(\mathcal{P})) \\
    &=|E(\mathcal{P})|+|V(\mathcal{P})|-2~\text{rank}(M(\mathcal{P}')) \\
    &=\big[|E(\mathcal{P}')|+|V(\mathcal{P}')|-2~\text{rank}(M(\mathcal{P}'))\big]+|E(\mathcal{P})|-|V(\mathcal{P})|+1 \\
    &=\ind\mathfrak{g}_C(\mathcal{P}')+|E(\mathcal{P})|-|V(\mathcal{P})|+1.
\end{align*}
The result follows.
\end{proof}

We are now in a position to prove Theorem~\ref{thm:index1}. Note that, considering Theorem~\ref{thm:disjoint}, it suffices to consider the case where $\mathcal{P}$ is connected.

\begin{theorem}\label{thm:index1c}
Let $\mathcal{P}$ be a connected, type-C poset of height one and $\mathfrak{g}=\mathfrak{g}_C(\mathcal{P})$. Then $$\ind\mathfrak{g}=|E(\mathcal{P})|-|V(\mathcal{P})|+2\delta_o,$$ where $\delta_o$ is the indicator function for $RG(\mathcal{P})$ containing no odd cycles.
\end{theorem}
\begin{proof}
There are five cases.
\bigskip

\noindent
\textbf{Case 1:} $RG(\mathcal{P})$ contains no dashed edge. In this case, the result follows from Theorem~\ref{thm:index01}.
\bigskip

\noindent
\textbf{Case 2:} $RG(\mathcal{P})$ contains no non-dashed edge, i.e., $\mathcal{P}$ is separable. In this case, applying Theorem~\ref{thm:indsep}, we find that
\begin{align*}
    \ind\mathfrak{g}_C(\mathcal{P})&=|E(\mathcal{P})|-|V(\mathcal{P})|+2 \\
    &=|E(\mathcal{P})|-|V(\mathcal{P})|+2\delta_0,
\end{align*}
where we have used the fact that $RG(\mathcal{P})$ cannot contain an odd cycle by Lemma~\ref{lem:nono} (d).
\bigskip

\noindent
\textbf{Case 3:} $RG(\mathcal{P})$ contains an odd cycle as well as both dashed and non-dashed edges. In this case, applying either Proposition~\ref{prop:sloopt} or Proposition~\ref{prop:ocyclet}, it follows that there exists a poset $\mathcal{P}'$ of height $(0,1)$ for which $RG(\mathcal{P}')$ contains an odd cycle, $|E(\mathcal{P})|=|E(\mathcal{P}')|$, $|V(\mathcal{P})|=|V(\mathcal{P}')|$, and $\ind\mathfrak{g}_C(\mathcal{P})=\ind\mathfrak{g}_C(\mathcal{P}')$. Now, by Theorem~\ref{thm:index01}, we have
\begin{align*}
    \ind\mathfrak{g}_C(\mathcal{P})&=\ind\mathfrak{g}_C(\mathcal{P}') \\
    &=|E(\mathcal{P}')|-|V(\mathcal{P}')| \\
    &=|E(\mathcal{P})|-|V(\mathcal{P})| \\
    &=|E(\mathcal{P})|-|V(\mathcal{P})|+2\delta_o,
\end{align*}
where we have used the fact that $RG(\mathcal{P})$ contains an odd cycle, i.e, $\delta_o=0$.
\bigskip

\noindent
\textbf{Case 4}: $RG(\mathcal{P})$ is a tree that contains both dashed and non-dashed edges. In this case, applying Proposition~\ref{prop:11t011} it follows that there exists a poset $\mathcal{P}'$ of height $(0,1)$ for which $RG(\mathcal{P}')$ is a tree, $|E(\mathcal{P})|=|E(\mathcal{P}')|$, $|V(\mathcal{P})|=|V(\mathcal{P}')|$, and $\ind\mathfrak{g}_C(\mathcal{P})=\ind\mathfrak{g}_C(\mathcal{P}')$. Now, by Theorem~\ref{thm:index01}, we have
\begin{align*}
\ind\mathfrak{g}_C(\mathcal{P})&=\ind\mathfrak{g}_C(\mathcal{P}') \\ &=|E(\mathcal{P}')|-|V(\mathcal{P}')|+2 \\
    &=|E(\mathcal{P})|-|V(\mathcal{P})|+2 \\
    &=|E(\mathcal{P})|-|V(\mathcal{P})|+2\delta_o,
\end{align*}
where we have used the fact that $RG(\mathcal{P})$ contains no odd cycles, i.e., $\delta_o=1$.
\bigskip

\noindent
\textbf{Case 5:} $RG(\mathcal{P})$ contains an even cycle, no odd cycles, and both dashed and non-dashed edges. In this case, applying Proposition~\ref{prop:evencl} it follows that there exists a poset $\mathcal{P}'$ of height $(0,1)$ for which $RG(\mathcal{P}')$ is a tree, $|V(\mathcal{P})|=|V(\mathcal{P}')|$, and $\ind\mathfrak{g}_C(\mathcal{P})=\ind\mathfrak{g}_C(\mathcal{P}')+|E(\mathcal{P})|-|V(\mathcal{P})|+1$. Note that since $RG(\mathcal{P}')$ is a tree with $|V(\mathcal{P}')|$ vertices, it follows that $|E(\mathcal{P}')|=|V(\mathcal{P}')|-1$. Now, by Theorem~\ref{thm:index01}, we have
\begin{align*}
    \ind\mathfrak{g}_C(\mathcal{P})&=\ind\mathfrak{g}_C(\mathcal{P}')+|E(\mathcal{P})|-|V(\mathcal{P})|+1 \\
    &=|E(\mathcal{P}')|-|V(\mathcal{P}')|+2+|E(\mathcal{P})|-|V(\mathcal{P})|+1 \\
    &=|V(\mathcal{P}')|-1-|V(\mathcal{P}')|+2+|E(\mathcal{P})|-|V(\mathcal{P})|+1 \\
    &=|E(\mathcal{P})|-|V(\mathcal{P})|+2 \\
    &=|E(\mathcal{P})|-|V(\mathcal{P})|+2\delta_o,
\end{align*}
where we have used the fact that $RG(\mathcal{P})$ contains no odd cycles, i.e., $\delta_o=1$.
\end{proof}

As noted above, combining Theorems~\ref{thm:disjoint} and~\ref{thm:index1c} establishes Theorem~\ref{thm:index1}. Using Theorem~\ref{thm:index1}, the characterization of Frobenius, type-C Lie poset algebras provided by Theorem 49 in \textbf{\cite{BCD}} can be extended mutatis mutandis.

\begin{theorem}\label{thm:frob}
If $\mathcal{P}$ is a type-C poset of height one, then $\mathfrak{g}_C(\mathcal{P})$ is Frobenius if and only if each connected component of $RG(\mathcal{P})$ contains a single cycle which consists of an odd number of vertices.
\end{theorem}

In the next section, we use Theorem~\ref{thm:index1} to help characterize type-C posets of height one for which $\mathfrak{g}_C(\mathcal{P})$ is contact.

\section{Contact Posets}\label{sec:con}

In this section, we characterize those type-B, C, and D posets of height one which correspond to contact Lie poset algebras. Ongoing, we refer to such posets as ``contact posets". As in Section~\ref{sec:indf}, for the sake of brevity, all results will concern type-C Lie poset algebras; considering Theorem~\ref{thm:onlyC}, though, all results still apply with ``type-C" replaced by ``type-B" or ``type-D". The main result of this section is Theorem~\ref{thm:contactchar} below.

\begin{theorem}\label{thm:contactchar}
Let $\mathcal{P}$ be a type-C poset of height one. Then $\mathcal{P}$ is contact if and only if
\begin{itemize}
    \item exactly one connected component of $RG(\mathcal{P})$ is a tree, and
    \item all remaining connected components contain a single cycle which consists of an odd number of vertices.
\end{itemize}
\end{theorem}

In order to prove Theorem~\ref{thm:contactchar}, we make use of an alternative characterization of contact Lie algebras. Let $\mathfrak{g}$ be a Lie algebra with ordered basis $\mathscr{B}(\mathfrak{g})=\{E_1,\hdots,E_n\}$. Recall that $\mathfrak{g}$ is contact only if it is odd-dimensional, so assume $\dim\mathfrak{g}=2k+1.$ Let $[I]=[E_1\hdots E_{2k+1}]^t$ and define $$\widehat{C}(\mathfrak{g},\mathscr{B}(\mathfrak{g}))=\begin{bmatrix}
0 & [I]^t\\
-[I] & C(\mathfrak{g},\mathscr{B}(\mathfrak{g}))
\end{bmatrix}.$$  
Take $\varphi\in\mathfrak{g}^*$. If $\{E_1^*,\dots,E_{2k+1}^*\}$ is the ``dual basis" associated to $\mathscr{B}(\mathfrak{g})$, i.e., $E^*_i(E_j)=\delta_{i=j}$ for $1\le i,j\le 2k+1$, then $\varphi$ can be written as a linear combination $\varphi=\sum_{i=1}^{2k+1}a_iE_i^*.$ In vector notation, $[\varphi]=[a_1,\dots,a_{2k+1}]^t.$ Applying $\varphi$ to each entry of $\widehat{C}(\mathfrak{g},\mathscr{B}(\mathfrak{g}))$ yields the $(2k+2)$-dimensional skew-symmetric matrix $$\varphi\left(\widehat{C}(\mathfrak{g},\mathscr{B}(\mathfrak{g}))\right)=\begin{bmatrix}
0 & [\varphi]^t\\
-[\varphi] & \varphi\left(C(\mathfrak{g},\mathscr{B}(\mathfrak{g})\right)
\end{bmatrix}.$$    

\noindent
Straightforward computations yield the following convenient characterization of contact Lie algebras.

\begin{theorem}[Salgado \textbf{\cite{InvCon}}]\label{thm:det} Let $\mathfrak{g}$ be an $n$-dimensional Lie algebra with basis $\mathscr{B}(\mathfrak{g})$ and $\varphi\in\mathfrak{g}^*$. If $n$ is odd, then $\mathfrak{g}$ is contact with contact form $\varphi$ if and only if $\det\varphi\left(\widehat{C}(\mathfrak{g},\mathscr{B}(\mathfrak{g}))\right)\neq 0$.
\end{theorem}

Let $\mathcal{P}$ be a type-C poset and $\mathfrak{g}=\mathfrak{g}_C(\mathcal{P})$. Ongoing, we will want to refer to certain rows of $\varphi(\widehat{C}(\mathfrak{g},\mathscr{B}_C(\mathcal{P})))$. Since throughout $\varphi$ will be clear from the context, it is omitted from the notation. We denote the first row by $\mathbf{I}(\mathcal{P})$. Note that the remaining rows correspond to rows of $\varphi(C(\mathfrak{g},\mathscr{B}(\mathfrak{g})))$ which are indexed by elements of $\mathscr{B}(\mathfrak{g})$. Consequently, we denote the row of $\varphi(\widehat{C}(\mathfrak{g},\mathscr{B}_C(\mathcal{P})))$ corresponding to the basis element
\begin{itemize}
    \item $D_i\in \mathscr{B}_C(\mathcal{P})$ by $\mathbf{\widehat{D}_i}(\mathcal{P})$,
    \item $E_{-i,i}\in \mathscr{B}_C(\mathcal{P})$ by $\mathbf{\widehat{E}_{-i,i}}(\mathcal{P})$,
    \item $R^{\pm}_{i,j}\in \mathscr{B}_C(\mathcal{P})$ by $\mathbf{\widehat{R}^{\pm}_{i,j}}(\mathcal{P})$, and
    \item $R_{i,j}\in \mathscr{B}_C(\mathcal{P})$ by $\mathbf{\widehat{R}_{i,j}},(\mathcal{P})$.
\end{itemize}

\begin{remark}\label{rem:nonzero}
Note that if $\mathcal{P}$ is a height-one type-C poset, then for $b\in\mathscr{B}_C(\mathcal{P})$ of the form $R^{\pm}_{i,j}$, $R_{i,j}$, or $E_{-i,i}$, the entries of $\mathbf{\widehat{b}}(\mathcal{P})$ are all multiples of $\varphi(b)$. Consequently, if $\varphi$ is a contact form, then $\varphi(b)\neq 0$ for all $b\in\mathscr{B}_C(\mathcal{P})$ of the form $R^{\pm}_{i,j}, R_{i,j},$ and $E_{-i,i}.$
\end{remark}

%\subsection{Characterization}

With the notation set, we proceed toward the proof of Theorem~\ref{thm:contactchar}. In Propositions~\ref{prop:nocycle1}, \ref{prop:noevencycle}, \ref{prop:noselfloops}, \ref{prop:noocycle1}, \ref{prop:noocycle2}, and \ref{prop:noocycle3} below, we show that if $\mathcal{P}$ is a connected, contact, type-C poset of height one, then $RG(\mathcal{P})$ cannot contain a cycle.

\begin{proposition}\label{prop:nocycle1}
Let $\mathcal{P}$ be a connected, type-C poset of height one. If $RG(\mathcal{P})$ contains either
\begin{enumerate}
    \item[\textup(a\textup)] a self-loop and no other cycles or
    \item[\textup(b\textup)] a single cycle that consists of an odd number of vertices,
\end{enumerate}
then $\mathfrak{g}_C(\mathcal{P})$ is not contact.
\end{proposition}
\begin{proof}
Applying Theorem~\ref{thm:index1}, $\ind\mathfrak{g}_C(\mathcal{P})=0\neq 1$. Since an algebra $\mathfrak{g}$ is contact only if $\ind\mathfrak{g}=1$, the result follows.
\end{proof}

For the cases when $RG(\mathcal{P})$ contains an even cycle or multiple odd cycles we require the following lemmas.

\begin{lemma}\label{lem:pathcon}
Let $n\ge 1$ and $1<i_0\neq i_1\neq \hdots\neq i_n\le N$. Define $$L_0=x_{1,N}+x_{i_0,N}+x_{i_{1},N}~\text{or}~x_{1,N}+x_{i_0,N}-x_{i_{1},N},$$
$$L_j=x_{1,N}+x_{i_j,N}+x_{i_{j+1},N},~x_{1,N}+x_{i_j,N}-x_{i_{j+1},N},~\text{or}~x_{1,N}-x_{i_j,N}+x_{i_{j+1},N}$$ for $1\le j\le n-2$, and $$L_{n-1}=x_{1,N}+x_{i_{n-1},N}+x_{i_{n},N}~\text{or}~x_{1,N}-x_{i_{n-1},N}+x_{i_{n},N}.$$ Suppose that $L_j=v_{1,N}+v_{i_j,N}-v_{i_{j+1},N}$ if and only if $L_{j+1}=v_{1,N}-v_{i_{j+1},N}+v_{i_{j+2},N}$ for $0\le j<n-1$. Then there exist constants $c_j\in\{-1,1\}$ for $0\le j<n$ such that $$\sum_{j=0}^{n-1}c_jL_j=\begin{cases}
x_{1,N}+x_{i_0,N}+x_{i_n,N}, & n\text{ is odd}\\
x_{i_0,N}-x_{i_n,N}, & n\text{ is even}.
\end{cases}$$
\end{lemma}
\begin{proof}
By induction on $n$. The cases $n=1$ or $n=2$ can be checked directly. Assume the result holds for $n-1\ge 2$. There are three cases.
\bigskip

\noindent
\textbf{Case 1:} $L_j=x_{1,N}+x_{i_j,N}+x_{i_{j+1},N}$ for $0\le j< n$. In this case,
\begin{align*}
    \sum_{j=0}^{n-1}(-1)^jL_j&=(-1)^{n-1}L_{n-1}+\sum_{j=0}^{n-2}(-1)^jL_j \\
    &=(-1)^{n-1}(x_{1,N}+x_{i_{n-1},N}+x_{i_{n},N})+\begin{cases}
        x_{1,N}+x_{i_0,N}+x_{i_{n-1},N}, & n-1\text{ is odd}\\
x_{i_0,N}-x_{i_{n-1},N}, & n-1\text{ is even}
    \end{cases} \\
    &=\begin{cases}
        -(x_{1,N}+x_{i_{n-1},N}+x_{i_{n},N})+x_{1,N}+x_{i_0,N}+x_{i_{n-1},N}, & n-1\text{ is odd}\\
(x_{1,N}+x_{i_{n-1},N}+x_{i_{n},N})+x_{i_0,N}-x_{i_{n-1},N}, & n-1\text{ is even}
    \end{cases} \\
    &=\begin{cases}
x_{i_0,N}-x_{i_n,N}, & n\text{ is even}\\
x_{1,N}+x_{i_0,N}+x_{i_n,N}, & n\text{ is odd},
\end{cases}
\end{align*} 
where the second equality follows from our induction hypothesis. So, taking $c_j=(-1)^j$ yields the result.
\bigskip

\noindent
\textbf{Case 2:} $L_j=x_{1,N}+(-1)^jx_{i_j,N}+(-1)^{j+1}x_{i_{j+1},N}$ for $0\le j<n$. Note that, in this case, $n$ must be even, and we have
\begin{align*}
    \sum_{j=0}^{n-1}(-1)^jL_j&=\sum_{j=0}^{\frac{n-2}{2}}(L_{2j}-L_{2j+1}) \\
    &=\sum_{j=0}^{\frac{n-2}{2}}[(x_{1,N}+x_{i_{2j},N}-x_{i_{2j+1},N})-(x_{1,N}-x_{i_{2j+1},N}+x_{i_{2j+2},N})] \\
    &=\sum_{j=0}^{\frac{n-2}{2}}(x_{i_{2j},N}-x_{i_{2j+2},N}) \\
    &=x_{i_0,N}-x_{i_n,N}.
\end{align*}
So, as in Case 1, taking $c_j=(-1)^j$ yields the result.
\bigskip

\noindent
\textbf{Case 3:} There exists $k$ such that $0\le k<n-2$ and either $$L_k=x_{1,N}+x_{i_{k},N}+x_{i_{k+1},N},\quad L_{k+1}=x_{1,N}+x_{i_{k+1},N}-x_{i_{k+2},N},\quad\text{and}\quad L_{k+2}=x_{1,N}-x_{i_{k+2},N}+x_{i_{k+3},N}$$ or $$L_k=x_{1,N}+x_{i_{k},N}-x_{i_{k+1},N},\quad L_{k+1}=x_{1,N}-x_{i_{k+1},N}+x_{i_{k+2},N},\quad\text{and}\quad L_{k+2}=x_{1,N}+x_{i_{k+2},N}+x_{i_{k+3},N}.$$ Without loss of generality, assume that $$L_k=x_{1,N}+x_{i_{0},N}+x_{i_{1},N},\quad L_{k+1}=x_{1,N}+x_{i_{1},N}-x_{i_{2},N},\quad\text{and}\quad L_{k+2}=x_{1,N}-x_{i_{2},N}+x_{i_{3},N}.$$ For $0\le j\le n-3$, define
\begin{align*}
    L'_j&=\begin{cases}
    L_k-L_{k+1}+L_{k+2}, & j=0 \\
    L_{j+2}, & 0<j<n-2
\end{cases} \\
&=\begin{cases}
    x_{1,N}+x_{i_0,N}+x_{i_{3},N}, & j=0 \\
    L_{j+2}, & 0<j<n-2.
\end{cases}
\end{align*}
Note that our induction hypothesis applies to the $L'_j$ for $0\le j\le n-3$. Thus, there exist constants $c'_j\in\{-1,1\}$ for $0\le j\le n-3$ such that $$\sum_{j=0}^{n-3}c'_jL'_j=\begin{cases}
x_{1,N}+x_{i_0,N}+x_{i_n,N}, & n-2\text{ is odd}\\
x_{i_0,N}-x_{i_n,N}, & n-2\text{ is even}.
\end{cases}$$ Now, setting $$c_j=\begin{cases}
    c'_0, & j=0,2 \\
    -c'_0, & j=1 \\
    c'_{j-2}, & 2<j\le n
\end{cases}$$
for $0\le j\le n$, we have that 
\begin{align*}
\sum_{j=0}^{n-1}c_jL_j&=c_0L_0+c_1L_1+c_2L_2+\sum_{j=3}^{n-1}c_jL_j \\
    &=c'_0L_0-c'_0L_1+c'_0L_2+\sum_{j=3}^{n-1}c_jL_j \\
    &=c'_0(L_0-L_1+L_2)+\sum_{j=3}^{n-1}c_jL_j \\
    &=c'_0L'_0+\sum_{j=1}^{n-3}c'_jL'_j \\
    &=\begin{cases}
x_{1,N}+x_{i_0,N}+x_{i_n,N}, & n-2\text{ is odd}\\
x_{i_0,N}-x_{i_n,N}, & n-2\text{ is even}
\end{cases}\\
&=\begin{cases}
x_{1,N}+x_{i_0,N}+x_{i_n,N}, & n\text{ is odd}\\
x_{i_0,N}-x_{i_n,N}, & n\text{ is even}.
\end{cases}.
\end{align*}
The result follows.
\end{proof}

\begin{lemma}\label{lem:conecyc}
Let $n\ge 2$ and $1<i_0\neq i_1\neq \hdots\neq i_n\le N$. Define $$L_0=x_{1,N}+x_{i_0,N}+x_{i_1,N}~\text{or}~x_{1,N}+x_{i_0,N}-x_{i_1,N},$$
$$L_j=x_{1,N}+x_{i_j,N}+x_{i_{j+1},N},~x_{1,N}+x_{i_j,N}-x_{i_{j+1},N},~\text{or}~x_{1,N}-x_{i_j,N}+x_{i_{j+1},N}$$ for $0\le j<n$, and $$L_n=x_{1,N}+x_{i_0,N}+x_{i_{n},N}~\text{or}~x_{1,N}+x_{i_0,N}-x_{i_{n},N}.$$ Suppose that
\begin{itemize}
    \item $L_j=x_{1,N}+x_{i_j,N}-x_{i_{j+1},N}$ if and only if $L_{j+1}=x_{1,N}-x_{i_{j+1},N}+x_{i_{j+2},N}$ for $0\le j<n-1$, and
    \item $L_{n-1}=x_{i,N}+x_{i_{n-1},N}-x_{i_n,N}$ if and only if $L_n=x_{1,N}-x_{i_n,N}+x_{i_{0},N}$.
\end{itemize}
Then there exists constants $c_j\in\{-1,1\}$ for $0\le j\le n$ such that $$\sum_{j=0}^nc_jL_j=\begin{cases}
x_{1,N}+2x_{i_0,N}, & n\text{ even} \\
0, & n\text{ odd}.
\end{cases}$$
\end{lemma}
\begin{proof}
By induction on $n$. The cases $n=2$ and $n=3$ can be checked directly. Assume the result holds for $n-1\ge 3$. There are three cases.
\bigskip

\noindent
\textbf{Case 1:} $L_j=x_{1,N}+x_{i_j,N}+x_{i_{j+1},N}$, for $0\le j<n$, and $L_n=x_{1,N}+x_{i_0,N}+x_{i_n,N}$. In this case,
\begin{align*}
    \sum_{j=0}^n(-1)^jL_j&=(-1)^nL_n+\sum_{j=0}^{n-1}(-1)^j(x_{1,N}+x_{i_j,N}+x_{i_{j+1},N}) \\
    &=(-1)^n(x_{1,N}+x_{i_0,N}+x_{i_n,N})+\sum_{j=0}^{n-1}(-1)^j(x_{1,N}+x_{i_j,N}+x_{i_{j+1},N}) \\
    &=(-1)^n(x_{i_0,N}+x_{i_n,N})+\sum_{j=0}^n(-1)^jx_{1,N}+\sum_{j=0}^{n-1}(-1)^j(x_{i_j,N}+x_{i_{j+1},N}) \\
    &=(-1)^n(x_{i_0,N}+x_{i_n,N})+\sum_{j=0}^n(-1)^jx_{1,N}+\sum_{j=0}^{n-1}(-1)^{j+1}(-x_{i_j,N}-x_{i_{j+1},N}) \\
    &=(-1)^n(x_{i_0,N}+x_{i_n,N})+\sum_{j=0}^n(-1)^jx_{1,N}+(x_{i_0,N}-(-1)^nx_{i_{n},N}) \\
    &=\begin{cases}
    x_{i_0,N}+x_{i_n,N}+x_{1,N}+x_{i_0,N}-x_{i_n,N}, & n\text{ even} \\
    -x_{i_0,N}-x_{i_n,N}+x_{i_0,N}+x_{i_n,N}, & n\text{ odd}
    \end{cases}\\
    &=\begin{cases}
    x_{1,N}+2x_{i_0,N}, & n\text{ even} \\
    0, & n\text{ odd},
    \end{cases}
\end{align*}
where the fifth equality follows from Lemma~\ref{lem:path}. So, taking $c_j=(-1)^j$ yields the result.
\bigskip

\noindent
\textbf{Case 2:} $L_j=~x_{1,N}+x_{i_j,N}-x_{i_{j+1},N}$ or $x_{1,N}-x_{i_j,N}+x_{i_{j+1},N}$ for $0\le j<n$ and $L_n=x_{1,N}+x_{i_0,N}-x_{i_{n},N}$. Note that this case can only occur when $n$ is odd. Consequently,
\begin{align*}
    \sum_{j=0}^n(-1)^jL_j&=(-1)^nL_n+\sum_{j=0}^{n-1}(-1)^jL_j \\
    &=(-1)^n(x_{1,N}+x_{i_0,N}-x_{i_n,N})+\sum_{j=0}^{n-1}(-1)^j(x_{1,N}+(-1)^jx_{i_j,N}+(-1)^{j+1}x_{i_{j+1},N}) \\
    &=\sum_{j=0}^n(-1)^jx_{1,N}-x_{i_0,N}+x_{i_n,N}+\sum_{j=0}^{n-1}(-1)^j((-1)^jx_{i_j,N}+(-1)^{j+1}x_{i_{j+1},N}) \\
    &=\sum_{j=0}^n(-1)^jx_{1,N}+\sum_{j=0}^n(x_{i_j,N}-x_{i_j,N}) \\
    &=0.
\end{align*}

So, as in Case 1, taking $c_j=(-1)^j$ yields the result.
\bigskip

\noindent
\textbf{Case 3:} There exists $0\le k<n$ such that $L_k=x_{1,N}-x_{i_k,N}+x_{i_{k+1},N}$ or $x_{1,N}+x_{i_k,N}-x_{i_{k+1},N}$ and either there exists $0\le j\neq k<n$ such that $L_j=x_{1,N}+x_{i_j,N}+x_{i_{j+1},N}$ or $L_n=x_{1,N}+x_{i_0,N}+x_{i_n,N}$. Without loss of generality, assume that $L_0=x_{1,N}+x_{i_0,N}+x_{i_1,N}$, $L_1=x_{1,N}+x_{i_1,N}-x_{i_2,N}$, and $L_2=x_{1,N}-x_{i_2,N}+x_{i_3,N}$. For $0\le j\le n-2$, define 
\begin{align*}
    L'_j&=\begin{cases}
    L_0-L_1+L_2, & j=0 \\
    L_{j+2}, & 0<j\le n-2
\end{cases} \\
&=\begin{cases}
    x_{1,N}+x_{i_0,N}+x_{i_3,N}, & j=0 \\
    L_{j+2}, & 0<j\le n-2.
\end{cases}
\end{align*} Note that our inductive hypothesis applies to the collection of vectors $L'_j$ for $0\le j\le n-2$. Thus, there exist constants $c'_j\in\{-1,1\}$ for $0\le j\le n-2$ such that $$\sum_{j=0}^{n-2}c'_jL'_j=\begin{cases}
x_{1,N}+2x_{i_0,N}, & n-2\text{ even} \\
0, & n-2\text{ odd}.
\end{cases}$$ Now, setting $$c_j=\begin{cases}
    c'_0, & j=0,2 \\
    -c'_0, & j=1 \\
    c'_{j-2}, & 2<j\le n
\end{cases}$$
for $0\le j\le n$, we have that 
\begin{align*}
\sum_{j=0}^nc_jL_j&=c_0L_0+c_1L_1+c_2L_2+\sum_{j=3}^nc_jL_j \\
    &=c'_0L_0-c'_0L_1+c'_0L_2+\sum_{j=3}^nc_jL_j \\
    &=c'_0(L_0-L_1+L_2)+\sum_{j=3}^nc_jL_j \\
    &=c'_0L'_0+\sum_{j=1}^{n-2}c'_jL'_j \\
    &=\begin{cases}
x_{1,N}+2x_{i_0,N}, & n-2\text{ even} \\
0, & n-2\text{ odd}
\end{cases} \\
&=\begin{cases}
x_{1,N}+2x_{i_0,N}, & n\text{ even} \\
0, & n\text{ odd}.
\end{cases}
\end{align*}
The result follows.
\end{proof}

\begin{lemma}\label{lem:conecyc2}
Let $n\ge 2$ even and $1<i_0\neq i_1\neq \hdots\neq i_n\le N$. Define $$L_0=x_{1,N}-x_{i_0,N}+x_{i_1,N},$$
$$L_j=x_{1,N}+x_{i_j,N}+x_{i_{j+1},N},~x_{1,N}+x_{i_j,N}-x_{i_{j+1},N},~\text{or}~x_{1,N}-x_{i_j,N}+x_{i_{j+1},N}$$ for $0\le j<n$, and $$L_n=x_{1,N}-x_{i_0,N}+x_{i_{n},N}.$$ Suppose that $L_j=x_{1,N}+x_{i_j,N}-x_{i_{j+1},N}$ if and only if $L_{j+1}=x_{1,N}-x_{i_{j+1},N}+x_{i_{j+2},N}$ for $1\le j<n-1$. Then there exists constants $c_j\in\{-1,1\}$ for $0\le j\le n$ such that $$\sum_{j=0}^nc_jL_j=-x_{1,N}+2x_{i_0,N}.$$
\end{lemma}
\begin{proof}
    Note that the collection of vectors $L_j$ for $1\le j\le n-1$ satisfy the hypotheses of Lemma~\ref{lem:pathcon}. So, applying Lemma~\ref{lem:pathcon}, there exists constants $c_j\in \{-1,1\}$ such that $$\sum_{j=1}^{n-1}c_jL_j=x_{1,N}+x_{i_1,N}+x_{i_n,N}.$$ Consequently, we have that 
    \begin{align*}
        -L_0-L_n+\sum_{j=1}^{n-1}c_jL_j&=-x_{1,N}+x_{i_0,N}-x_{i_1,N}-x_{1,N}+x_{i_0,N}-x_{i_n,N}+x_{1,N}+x_{i_1,N}+x_{i_n,N} \\
        &=-x_{1,N}+2x_{i_0,N},
    \end{align*}
    as desired.
\end{proof}

\begin{proposition}\label{prop:noevencycle}
 Let $\mathcal{P}$ be a connected, type-C poset of height one. If $RG(\mathcal{P})$ contains an even cycle, then $\mathcal{P}$ is not contact.
\end{proposition}
\begin{proof}
Assume that $\mathfrak{g}=\mathfrak{g}_C(\mathcal{P})$ is contact and fix a choice of contact form $\varphi\in\mathfrak{g}^*$. Set $N=|V(\mathcal{P})|+|E(\mathcal{P})|+1$ and let $\mathcal{C}$ denote an even cycle of $RG(\mathcal{P})$. We assume that $\mathcal{C}$ is defined by the sequence of vertices $p_0,p_1,\hdots,p_{2k+1},p_0$ where if $\{p_0,p_1\}$ (resp., $\{p_{2k+1},p_0\}$) is dashed, then $p_0>p_1$ (resp., $p_0>p_{2k+1})$. For $0\le j\le 2k+1$, let $R_j$ (resp., $\mathbf{\widehat{R}}_j$) denote the element of $\mathscr{B}_C(\mathcal{P})$ (resp., row of $\varphi(\widehat{C}(\mathfrak{g},\mathscr{B}_C(\mathcal{P})))$) corresponding to $\{p_j,p_{j+1}\}$ when $0\le j<2k+1$ and $\{p_{2k+1},p_0\}$ when $j=2k+1$, i.e.,
$$R_j=\begin{cases}
    R_{p_j,p_{j+1}}, & 0\le j<2k+1,~\{p_j,p_{j+1}\}~\text{is dashed, and}~p_{j}<p_{j+1} \\
    R_{p_{j+1},p_{j}}, & 0\le j<2k+1,~\{p_j,p_{j+1}\}~\text{is dashed, and}~p_{j}>p_{j+1} \\
    R^{\pm}_{p_j,p_{j+1}}, & 0\le j<2k+1~\text{and}~\{p_j,p_{j+1}\}~\text{is non-dashed} \\
    R_{p_{2k+1},p_0}, & j=2k+1~\text{and}~\{p_0,p_{2k+1}\}~\text{is dashed} \\
    R^{\pm}_{p_{2k+1},p_0}, & j=2k+1~\text{and}~\{p_0,p_{2k+1}\}~\text{is non-dashed}
\end{cases}$$ and similarly for $\mathbf{\widehat{R}}_j$. Order the elements of $\mathscr{B}_C(\mathcal{P})$ so that
\begin{itemize}
    \item $D_{p_i}$ for $0\le i\le 2k+1$ occur first, listed in increasing order of $i$ followed by
    \item $D_{i}$ for $i\in\mathcal{P}^+\backslash\{p_0,\hdots,p_{2k+1}\}$ in increasing order of $i$ in $\mathbb{Z}$ followed by
    \item $E_{-i,i}$ in increasing order of $i$ in $\mathbb{Z}$ followed by
    \item $R^{\pm}_{i,j}$ for $i<j$ such that $-i\prec j$ and $-j\prec i$ in increasing lexicographic order of $(i,j)$ followed by
    \item $R_{i,j}$ for $i<j$ such that $-j\prec -i$ and $i\prec j$ in increasing lexicographic order of $(i,j)$.
\end{itemize}
With this ordering of $\mathscr{B}_C(\mathcal{P})$, for $0\le j\le 2k+1$, we have that $$\mathbf{\widehat{R}}_j=\begin{cases}
    -\varphi(R_j)(x_{1,N}-x_{j+2,N}+x_{j+3,N}), & 0\rn{<}  j<2k+1,~\{p_j,p_{j+1}\}~\text{is dashed, and}~p_{j}<p_{j+1} \\
    -\varphi(R_j)(x_{1,N}+x_{j+2,N}-x_{j+3,N}), & 0\le j<2k+1,~\{p_j,p_{j+1}\}~\text{is dashed, and}~p_{j}>p_{j+1} \\
    -\varphi(R_j)(x_{1,N}+x_{j+2,N}+x_{j+3,N}), & 0\le j<2k+1~\text{and}~\{p_j,p_{j+1}\}~\text{is non-dashed} \\
    -\varphi(R_j)(x_{1,N}+x_{2,N}-x_{2k+3,N}), & j=2k+1~\text{and}~\{p_0,p_{2k+1}\}~\text{is dashed} \\
    -\varphi(R_j)(x_{1,N}+x_{2,N}+x_{2k+3,N}), & j=2k+1~\text{and}~\{p_0,p_{2k+1}\}~\text{is non-dashed}.
\end{cases}$$
Since $\varphi$ is a contact form, considering Remark~\ref{rem:nonzero} we have that $\varphi(R_j)\neq 0$ for $0\le j\le 2k+1$. Thus, we can define the collection of vectors $L_j=\frac{1}{-\varphi(R_j)}\mathbf{\widehat{R}}_j$ for $0\le j\le 2k+1$. Considering our assumptions on the edges of $\mathcal{C}$ along with Lemma~\ref{lem:nono}, it is straightforward to verify that the collection of vectors $L_j$ for $0\le j\le 2k+1$ satisfy the hypotheses of Lemma~\ref{lem:conecyc}. Consequently, applying Lemma~\ref{lem:conecyc}, there exist constants $c_j\in\{-1,1\}$ such that $\sum_{j=0}^{2k+1}c_jL_j=0$; but this implies that $\varphi\left(\widehat{C}(\mathfrak{g},\mathscr{B}(\mathfrak{g}))\right)$ does not have full rank, which contradicts that $\varphi$ is a contact form. Therefore, $\mathfrak{g}_C(\mathcal{P})$ is not contact.
\end{proof}

% \bl{Propositions~\ref{prop:noselfloops} through~\ref{prop:noocycle3} below cover the case where $RG(\mathcal{P})$ contains multiple odd cycles.}

\begin{proposition}\label{prop:noselfloops}
Let $\mathcal{P}$ be a connected, type-C poset of height one. If $RG(\mathcal{P})$ contains two self-loops, then $\mathfrak{g}_C(\mathcal{P})$ is not contact.
\end{proposition}
\begin{proof}
Assume that $\mathfrak{g}=\mathfrak{g}_C(\mathcal{P})$ is contact and fix a choice of contact form $\varphi\in\mathfrak{g}^*$. Set $N=|V(\mathcal{P})|+|E(\mathcal{P})|+1$ and let $p$ and $q$ denote two distinct vertices of $RG(\mathcal{P})$ that define self-loops. Since $RG(\mathcal{P})$ is connected, there exists a path in $RG(\mathcal{P})$ between $p$ and $q$ defined, say, by the sequence of vertices $p=p_0,p_1,\hdots,p_{k-1},p_k=q$. For $0\le j\le k-1$, let $R_j$ (resp., $\mathbf{\widehat{R}}_j$) denote the element of $\mathscr{B}_C(\mathcal{P})$ (resp., row of $\varphi(\widehat{C}(\mathfrak{g},\mathscr{B}_C(\mathcal{P})))$) corresponding to $\{p_j,p_{j+1}\}$, i.e.,
$$R_j=\begin{cases}
    R_{p_j,p_{j+1}}, & \{p_j,p_{j+1}\}~\text{is dashed and}~p_{j}<p_{j+1} \\
    R_{p_{j+1},p_{j}}, & \{p_j,p_{j+1}\}~\text{is dashed and}~p_{j}>p_{j+1} \\
    R^{\pm}_{p_j,p_{j+1}}, & \{p_j,p_{j+1}\}~\text{is non-dashed}
\end{cases}$$ and similarly for $\mathbf{\widehat{R}}_j$. Order the elements of $\mathscr{B}_{C}(\mathcal{P})$ so that
\begin{itemize}
    \item $D_{p_i}$ for $0\le i\le k$ occur first, listed in increasing order of $i$ followed by
    \item $D_{p}$ for $p\in\mathcal{P}^+\backslash\{p_0,\hdots,p_k\}$ in increasing order of $p$ in $\mathbb{Z}$ followed by
    \item $E_{-i,i}$ in increasing order of $i$ in $\mathbb{Z}$ followed by
    \item $R^{\pm}_{i,j}$ for $i<j$ such that $-i\prec j$ and $-j\prec i$ in increasing lexicographic order of $(i,j)$ followed by
    \item $R_{i,j}$ for $i<j$ such that $-j\prec -i$ and $i\prec j$ in increasing lexicographic order of $(i,j)$.
\end{itemize}
With this ordering of $\mathscr{B}_C(\mathcal{P})$, we have
$$\mathbf{\widehat{E}_{-p_0,p_0}(\mathcal{P})}=-\varphi(E_{-p_0,p_0})(x_{1,N}+2x_{2,N}),$$
$$\mathbf{\widehat{E}_{-p_k,p_k}(\mathcal{P})}=-\varphi(E_{-p_k,p_k})(x_{1,N}+2x_{k+2,N}),$$ and $$\mathbf{\widehat{R}}_j=\begin{cases}
    -\varphi(R_j)(x_{1,N}-x_{j+2,N}+x_{j+3,N}), & \{p_j,p_{j+1}\}~\text{is dashed and}~p_{j}<p_{j+1} \\
    -\varphi(R_j)(x_{1,N}+x_{j+2,N}-x_{j+3,N}), & \{p_j,p_{j+1}\}~\text{is dashed and}~p_{j}>p_{j+1} \\
    -\varphi(R_j)(x_{1,N}+x_{j+2,N}+x_{j+3,N}), & \{p_j,p_{j+1}\}~\text{is non-dashed}
\end{cases}$$
for $0\le j\le k-1$. Since $\varphi$ is a contact form, considering Remark~\ref{rem:nonzero}, we have that $\varphi(E_{-p_0,p_0}),\varphi(E_{-p_k,p_k})\neq 0$ and $\varphi(R_j)\neq 0$ for $0\le j\le k-1$. Thus, we can define the collection of vectors $L_j=\frac{1}{-\varphi(R_j)}\mathbf{\widehat{R}}_j$ for $0\le j\le k-1$. Considering Lemma~\ref{lem:nono}, it is straightforward to verify that the collection of vectors $L_j$ for $0\le j\le k-1$ satisfy the hypotheses of Lemma~\ref{lem:pathcon}. Consequently, applying Lemma~\ref{lem:pathcon}, there exist constants $c_j\in\{-1,1\}$ such that $$\sum_{j=0}^{k-1}c_jL_j=\begin{cases}
x_{1,N}+x_{2,N}+x_{k+2,N}, & k\text{ is odd} \\
x_{2,N}-x_{k+2,N}, & k\text{ is even}
\end{cases}$$ so that $$\frac{1}{2\varphi(E_{-p_0,p_0})}\mathbf{\widehat{E}_{p_0,p_0}(\mathcal{P})}+(-1)^{k+1}\frac{1}{2\varphi(E_{-p_k,p_k})}\mathbf{\widehat{E}_{p_k,p_k}(\mathcal{P})}+\sum_{j=0}^{k-1}c_jL_j=0;$$
but this implies that $\varphi\left(\widehat{C}(\mathfrak{g},\mathscr{B}(\mathfrak{g}))\right)$ does not have full rank, contradicting that $\varphi$ is a contact form. Therefore, $\mathfrak{g}_C(\mathcal{P})$ is not contact.
\end{proof}

\begin{proposition}\label{prop:noocycle1}
Let $\mathcal{P}$ be a connected, type-C poset of height one. If $RG(\mathcal{P})$ contains two odd cycles that share more than one vertex, then $\mathfrak{g}_C(\mathcal{P})$ is not contact.
\end{proposition}
\begin{proof}
    Let $\mathcal{C}_1$ and $\mathcal{C}_2$ be two odd cycles in $RG(\mathcal{P})$ that share more than one vertex. We claim that $RG(\mathcal{P})$ contains an even cycle. Assume $\mathcal{C}_i$ consists of $2k_i+1$ vertices with $k_i\in\mathbb{Z}_{> 0}$ for $i=1,2$. Starting from one intersection point between the two, say $p$, move along $\mathcal{C}_1$ until reaching another intersection point, say $q$, and call the resulting path $P_1$. Now, there must exist a path within $\mathcal{C}_2$ connecting $p$ and $q$ that contains no other intersection points with $\mathcal{C}_1$, call such a path $P_2$. Form the subgraph $\mathcal{C}_3$ of $RG(\mathcal{P})$ whose edges are those of $P_1$ and $P_2$. Note that either 
    \begin{enumerate}
        \item[(1)] $P_1=P_2$ so that $\mathcal{C}_3$ is the graph consisting of $p$, $q$, and the edge between them or
        \item[(2)] $\mathcal{C}_3$ is a cycle whose intersection with $\mathcal{C}_i$ is $P_i$ for $i=1,2$.
    \end{enumerate}
    In case (1), consider the subgraph $\mathcal{C}_4$ of $RG(\mathcal{P})$ induced by all edges of both $\mathcal{C}_1$ and $\mathcal{C}_2$ except the shared edge between $p$ and $q$. Note that $\mathcal{C}_4$ is a cycle containing $$(2k_1+1)+(2k_2+1)-2=2(k_1+k_2)$$ vertices, i.e., $\mathcal{C}_4$ is an even cycle and the claim follows. In case (2), we may assume that $\mathcal{C}_3$ is an odd cycle containing $2k_3+1$ vertices for $k_3\in\mathbb{Z}_{> 0}$. Now, if $P_1$ contains $t$ vertices for $t\in \mathbb{Z}_{>1}$, then consider the subgraph $\mathcal{C}_4$ of $RG(\mathcal{P})$ induced by all edges of both $\mathcal{C}_1$ and $\mathcal{C}_3$ except those corresponding to $P_1$. Note that $\mathcal{C}_4$ is a cycle containing $$(2k_1+1)+(2k_3+1)-2(t-2)-2=2(k_1+k_3-t+2)$$ vertices, i.e., $\mathcal{C}_4$ is an even cycle and the claim follows. We have shown that, in both cases, $RG(\mathcal{P})$ contains an even cycle; therefore, $\mathfrak{g}_C(\mathcal{P})$ is not contact, by Proposition~\ref{prop:noevencycle}.
\end{proof}

\begin{proposition}\label{prop:noocycle2}
Let $\mathcal{P}$ be a connected, type-C poset of height one. If $RG(\mathcal{P})$ contains two odd cycles that share exactly one vertex, then $\mathfrak{g}_C(\mathcal{P})$ is not contact.
\end{proposition}
\begin{proof}
    Assume that $\mathfrak{g}=\mathfrak{g}_C(\mathcal{P})$ is contact and fix a contact form $\varphi\in\mathfrak{g}^*$. Set $N=|V(\mathcal{P})|+|E(\mathcal{P})|+1$. Let $\mathcal{C}_1$ and $\mathcal{C}_2$ denote the two odd cycles, and $p$ denote the shared vertex. Throughout, for $i=1,2$, we assume that if $\mathcal{C}_i$ is not a self-loop, then $\mathcal{C}_i$ is defined by the sequence of vertices $p=p^i_0,p^i_1,\hdots,p^i_{k_i},p^i_0=p$. Moreover, for $i=1,2$ and $0\le j\le k_i$, we let $R^i_j$ (resp., $\mathbf{\widehat{R}}^i_j$) denote the element of $\mathscr{B}_C(\mathcal{P})$ (resp., row of $\varphi(\widehat{C}(\mathfrak{g},\mathscr{B}_C(\mathcal{P})))$) corresponding to $\{p^i_j,p^i_{j+1}\}$ when $0\le j<k_i$ and $\{p,p^i_k\}$ when $j=k_i$, i.e.,
$$R^i_j=\begin{cases}
    R_{p^i_j,p^i_{j+1}}, & 0\le j<k_i,~\{p^i_j,p^i_{j+1}\}~\text{is dashed, and}~p^i_{j}<p^i_{j+1} \\
    R_{p^i_{j+1},p^i_{j}}, & 0\le j<k_i,~\{p^i_j,p^i_{j+1}\}~\text{is dashed, and}~p^i_{j}>p^i_{j+1} \\
    R^{\pm}_{p^i_j,p^i_{j+1}}, & 0\le j<k_i~\text{and}~\{p^i_j,p^i_{j+1}\}~\text{is non-dashed} \\
    R_{p,p^i_{k_i}}, & j=k_i,~\{p,p^i_{k_i}\}~\text{is dashed, and}~p<p^i_{k_i} \\
    R_{p^i_{k_i},p}, & j=k_i,~\{p,p^i_{k_i}\}~\text{is dashed, and}~p>p^i_{k_i} \\
    R^{\pm}_{p,p^i_{k_i}}, & j=k_i~\text{and}~\{p,p^i_{k_i}\}~\text{is non-dashed}
\end{cases}$$ and similarly for $\mathbf{\widehat{R}}^i_j$. Order the elements of $\mathscr{B}_{C}(\mathcal{P})$ so that
\begin{itemize}
    \item $D_{p^1_j}$, for $0\le j\le k_1$, occur first listed in increasing order of $j$ followed by
    \item $D_{p^2_j}$, for $1\le j\le k_2$, listed in increasing order of $j$ followed by
    \item $D_{q}$, for $q\in\mathcal{P}^+\backslash\{p^1_0,\hdots,p^1_{k_1},p^2_0,\hdots,p^2_{k_2}\}$, in increasing order of $q$ in $\mathbb{Z}$ followed by
    \item $R^{\pm}_{i,j}$, for $i<j$, such that $-i\prec j$ and $-j\prec i$ in increasing lexicographic order of $(i,j)$ followed by
    \item $R_{i,j}$, for $i<j$, such that $-j\prec -i$ and $i\prec j$ in increasing lexicographic order of $(i,j)$. 
\end{itemize}
With this ordering, if $\mathcal{C}_1$ is not a self-loop, then $$\mathbf{\widehat{R}}^1_j=\begin{cases}
    -\varphi(R^1_j)(x_{1,N}-x_{j+2,N}+x_{j+3,N}), & 0\le j<k_1,~\{p^1_j,p^1_{j+1}\}~\text{is dashed, and}~p^1_{j}<p^1_{j+1} \\
    -\varphi(R^1_j)(x_{1,N}+x_{j+2,N}-x_{j+3,N}), & 0\le j<k_1,~\{p^1_j,p^1_{j+1}\}~\text{is dashed, and}~p^1_{j}>p^1_{j+1} \\
    -\varphi(R^1_j)(x_{1,N}+x_{j+2,N}+x_{j+3,N}), & 0\le j<k_1~\text{and}~\{p^1_j,p^1_{j+1}\}~\text{is non-dashed} \\
    -\varphi(R^1_j)(x_{1,N}-x_{2,N}+x_{k_1+2,N}), & j=k_1,~\{p,p^1_{k_1}\}~\text{is dashed, and}~p<p_{k_1}^1 \\
    -\varphi(R^1_j)(x_{1,N}+x_{2,N}-x_{k_1+2,N}), & j=k_1,~\{p,p^1_{k_1}\}~\text{is dashed, and}~p>p_{k_1}^1 \\
    -\varphi(R^1_j)(x_{1,N}+x_{2,N}+x_{k_1+2,N}), & j=k_1~\text{and}~\{p,p^1_{k_1}\}~\text{is non-dashed}
\end{cases}$$
for $0\le j\le k_1$. Similarly, if $\mathcal{C}_2$ is not a self-loop, then $$\mathbf{\widehat{R}}^2_j=\begin{cases}
    -\varphi(R^2_j)(x_{1,N}-x_{2,N}+x_{k_1+3,N}), & j=0,~\{p,p^2_1\}~\text{is dashed, and}~p<p^2_1 \\
    -\varphi(R^2_j)(x_{1,N}+x_{2,N}-x_{k_1+3,N}), & j=0,~\{p,p^2_1\}~\text{is dashed, and}~p>p^2_1 \\
    -\varphi(R^2_j)(x_{1,N}+x_{2,N}+x_{k_1+3,N}), & j=0~\text{and}~\{p,p^2_{1}\}~\text{is non-dashed} \\
    -\varphi(R^2_j)(x_{1,N}-x_{j+k_1+2,N}+x_{j+k_1+3,N}), & 0< j<k_2,~\{p^2_j,p^2_{j+1}\}~\text{is dashed, and}~p^2_{j}<p^2_{j+1} \\
    -\varphi(R^2_j)(x_{1,N}+x_{j+k_1+2,N}-x_{j+k_1+3,N}), & 0< j<k_2,~\{p^2_j,p^2_{j+1}\}~\text{is dashed, and}~p^2_{j}>p^2_{j+1} \\
    -\varphi(R^2_j)(x_{1,N}+x_{j+k_1+2,N}+x_{j+k_1+3,N}), & 0< j<k_2~\text{and}~\{p^2_j,p^2_{j+1}\}~\text{is non-dashed} \\
    -\varphi(R^2_j)(x_{1,N}-x_{2,N}+x_{k_1+k_2+2,N}), & j=k_2,~\{p,p^2_{k_2}\}~\text{is dashed, and}~p<p_{k_2}^2 \\
    -\varphi(R^2_j)(x_{1,N}+x_{2,N}-x_{k_1+k_2+2,N}), & j=k_2,~\{p,p^2_{k_2}\}~\text{is dashed, and}~p>p_{k_2}^2 \\
    -\varphi(R^2_j)(x_{1,N}+x_{2,N}+x_{k_1+k_2+2,N}), & j=k_2~\text{and}~\{p,p^2_{k_2}\}~\text{is non-dashed}
\end{cases}$$ for $0\le j\le k_2$. Finally, if one of $\mathcal{C}_1$ or $\mathcal{C}_2$ is a self-loop, then $$\mathbf{\widehat{E}_{-p,p}(\mathcal{P})}=-\varphi(E_{-p,p})(x_{1,N}+2x_{2,N}).$$ Considering Lemma~\ref{lem:nono}, there are two cases.
\bigskip

\noindent
\textbf{Case 1:} $\mathcal{C}_1$ or $\mathcal{C}_2$ is a self-loop. Without loss of generality, assume that $\mathcal{C}_2$ is a self-loop. Since $\varphi$ is a contact form, considering Remark~\ref{rem:nonzero}, we have that $\varphi(E_{-p,p})\neq 0$ and $\varphi(R^1_j)\neq 0$ for $0\le j\le k_1$. Thus, we can define the collection of vectors $L_j=\frac{1}{-\varphi(R^1_j)}\mathbf{\widehat{R}}^1_j$ for $0\le j\le k_1$. Considering Lemma~\ref{lem:nono}, it is straightforward to verify that the collection of vectors $L_j$ for $0\le j\le k_1$ satisfy the hypotheses of Lemma~\ref{lem:conecyc}. Since $k_1$ is even by assumption, applying Lemma~\ref{lem:conecyc}, we find that there exist constants $c_j\in \{-1,1\}$ for $0\le j\le k_1$ such that $$\sum_{j=0}^{k_1}c_jL_j=x_{1,N}+2x_{2,N}.$$ However, this implies that $$\frac{1}{\varphi(E_{-p,p})}\mathbf{\widehat{E}_{-p,p}(\mathcal{P})}+\sum_{j=0}^{k_1}c_jL_j=0,$$ i.e., $\varphi\left(\widehat{C}(\mathfrak{g},\mathscr{B}(\mathfrak{g}))\right)$ does not have full rank, contradicting that $\varphi$ is a contact form. Therefore, $\mathfrak{g}_C(\mathcal{P})$ is not contact.
\bigskip

\noindent
\textbf{Case 2:} Neither $\mathcal{C}_1$ nor $\mathcal{C}_2$ is a self-loop and one of the following holds:
\begin{enumerate}
    \item[(1)] at least one of $\{p,p_1^1\}$, $\{p,p_1^2\}$, $\{p,p_{k_1}^1\}$, and $\{p,p_{k_2}^2\}$ is non-dashed;
    \item[(2)] $\{p,p_1^1\}$, $\{p,p_1^2\}$, $\{p,p_{k_1}^1\}$, and $\{p,p_{k_2}^2\}$ are all dashed and $p^1_1,p_1^2<p>p^1_{k_1},p^2_{k_2}$; or
    \item[(3)] $\{p,p_1^1\}$, $\{p,p_1^2\}$, $\{p,p_{k_1}^1\}$, and $\{p,p_{k_2}^2\}$ are all dashed and $p^1_1,p_1^2>p<p^1_{k_1},p^2_{k_2}$.
\end{enumerate}
Since $\varphi$ is a contact form, considering Remark~\ref{rem:nonzero} we have that $\varphi(R^i_j)\neq 0$ for $i=1,2$ and $0\le j\le k_i$. Thus, we can define the collection of vectors $L^i_j=\frac{1}{-\varphi(R^i_j)}\mathbf{\widehat{R}}^i_j$ for $i=1,2$ and $0\le j\le k_i$. Considering Lemma~\ref{lem:nono}, in cases (1) and (2) it is straightforward to verify that, for $i=1,2$, the collection of vectors $L^i_j$ for $0\le j\le k_i$ satisfy the hypotheses of Lemma~\ref{lem:conecyc}. On the other hand, in case (3), the collections of vectors satisfy the hypotheses of Lemma~\ref{lem:conecyc2}. Since $k_1$ and $k_2$ are even by assumption, in cases (1) and (2) (resp., case (3)), we apply Lemma~\ref{lem:conecyc} (resp., Lemma~\ref{lem:conecyc2}) to find that there exist constants $c^i_j\in \{-1,1\}$ for $i=1,2$ and $0\le j\le k_1$ such that $$\sum_{j=0}^{k_i}c^i_jL^i_j=x_{1,N}+2x_{2,N}\quad\left(\text{resp.,}~\sum_{j=0}^{k_i}c^i_jL^i_j=-x_{1,N}+2x_{2,N}\right).$$ In all cases, we have that $$\sum_{j=0}^{k_1}c^1_jL^1_j-\sum_{j=0}^{k_2}c^2_jL^2_j=0,$$ which implies that $\varphi\left(\widehat{C}(\mathfrak{g},\mathscr{B}(\mathfrak{g}))\right)$ does not have full rank, contradicting that $\varphi$ is a contact form. Therefore, $\mathfrak{g}_C(\mathcal{P})$ is not contact.
\end{proof}

\begin{proposition}\label{prop:noocycle3}
Let $\mathcal{P}$ be a connected, type-C poset of height one. If $RG(\mathcal{P})$ contains two disjoint odd cycles, then $\mathfrak{g}_C(\mathcal{P})$ is not contact.
\end{proposition}
\begin{proof}
Assume that $\mathfrak{g}=\mathfrak{g}_C(\mathcal{P})$ is contact and fix a contact form $\varphi\in\mathfrak{g}^*$. Set $N=|V(\mathcal{P})|+|E(\mathcal{P})|+1$. Let $\mathcal{C}_1$ and $\mathcal{C}_2$ denote the two odd cycles. Since $RG(\mathcal{P})$ is connected, there exists a path, say $T$, that connects a vertex $p_0^1$ of $\mathcal{C}_1$ to a vertex $p_0^2$ of $\mathcal{C}_2$ and contains no other vertices of either cycle. For $i=1,2$, if $\mathcal{C}_i$ is not a self-loop, then assume that $\mathcal{C}_i$ is defined by the sequence of vertices $p_0^i,p_1^i,\hdots,p_{k_i}^i,p_0^i$. Moreover, we assume that $T$ is defined by the sequence of vertices $p_0^1=p_0^3,p_1^3,\hdots,p_{k_3}^3=p_0^2$. For $i=1,2$ and $0\le j\le k_i$, we let $R^i_j$ (resp., $\mathbf{\widehat{R}}^i_j$) denote the element of $\mathscr{B}_C(\mathcal{P})$ (resp., row of $\varphi(\widehat{C}(\mathfrak{g},\mathscr{B}_C(\mathcal{P})))$) corresponding to $\{p^i_j,p^i_{j+1}\}$ when $0\le j<k_i$ and $\{p,p^i_k\}$ when $j=k_i$, i.e.,
$$R^i_j=\begin{cases}
    R_{p^i_j,p^i_{j+1}}, & 0\le j<k_i,~\{p^i_j,p^i_{j+1}\}~\text{is dashed, and}~p^i_{j}<p^i_{j+1} \\
    R_{p^i_{j+1},p^i_{j}}, & 0\le j<k_i,~\{p^i_j,p^i_{j+1}\}~\text{is dashed, and}~p^i_{j}>p^i_{j+1} \\
    R^{\pm}_{p^i_j,p^i_{j+1}}, & 0\le j<k_i~\text{and}~\{p^i_j,p^i_{j+1}\}~\text{is non-dashed} \\
    R_{p^i_0,p^i_{k_i}}, & j=k_i,~\{p^i_0,p^i_{k_i}\}~\text{is dashed, and}~p^i_0<p^i_{k_i} \\
    R_{p^i_{k_i},p^i_0}, & j=k_i,~\{p^i_0,p^i_{k_i}\}~\text{is dashed, and}~p^i_0>p^i_{k_i} \\
    R^{\pm}_{p^i_0,p^i_{k_i}}, & j=k_i~\text{and}~\{p^i_0,p^i_{k_i}\}~\text{is non-dashed}
\end{cases}$$ and similarly for $\mathbf{\widehat{R}}^i_j$. For $0\le j< k_3$, we let $R^3_j$ (resp., $\mathbf{\widehat{R}}^3_j$) denote the element of $\mathscr{B}_C(\mathcal{P})$ (resp., row of $\varphi(\widehat{C}(\mathfrak{g},\mathscr{B}_C(\mathcal{P})))$) corresponding to $\{p^3_j,p^3_{j+1}\}$, i.e., $$R^3_j=\begin{cases}
    R_{p^3_j,p^3_{j+1}}, & 0\le j<k_3,~\{p^3_j,p^3_{j+1}\}~\text{is dashed, and}~p^3_{j}<p^3_{j+1} \\
    R_{p^3_{j+1},p^3_{j}}, & 0\le j<k_3,~\{p^3_j,p^3_{j+1}\}~\text{is dashed, and}~p^3_{j}>p^i_{j+1} \\
    R^{\pm}_{p^3_j,p^3_{j+1}}, & 0\le j<k_3~\text{and}~\{p^3_j,p^3_{j+1}\}~\text{is non-dashed}
\end{cases}$$ and similarly for $\mathbf{\widehat{R}}^3_j$. Order the elements of $\mathscr{B}_{C}(\mathcal{P})$ so that
\begin{itemize}
    \item $D_{p^1_j}$, for $0\le j\le k_1$, occur first listed in increasing order of $j$ followed by
    \item $D_{p^3_j}$, for $1\le j\le k_3$, in increasing order of $j$ followed by
    \item $D_{p^2_j}$, for $1\le j\le k_2$, in increasing order of $j$ followed by
    \item $D_{p}$, for $p\in\mathcal{P}^+\backslash\{p^1_0,\hdots,p^1_{k_1},p^2_0,\hdots,p^2_{k_2},p^3_0,\hdots,p^3_{k_3}\}$, in increasing order of $p$ in $\mathbb{Z}$ followed by
    \item $R^{\pm}_{i,j}$, for $i<j$ such that $-i\prec j$ and $-j\prec i$, in increasing lexicographic order of $(i,j)$ followed by
    \item $R_{i,j}$, for $i<j$ such that $-j\prec -i$ and $i\prec j$, in increasing lexicographic order of $(i,j)$. 
\end{itemize}
With this ordering, if $\mathcal{C}_1$ is not a self-loop, then $$\mathbf{\widehat{R}}^1_j=\begin{cases}
    -\varphi(R^1_j)(x_{1,N}-x_{j+2,N}+x_{j+3,N}), & 0\le j<k_1,~\{p^1_j,p^1_{j+1}\}~\text{is dashed, and}~p^1_{j}<p^1_{j+1} \\
    -\varphi(R^1_j)(x_{1,N}+x_{j+2,N}-x_{j+3,N}), & 0\le j<k_1,~\{p^1_j,p^1_{j+1}\}~\text{is dashed, and}~p^1_{j}>p^1_{j+1} \\
    -\varphi(R^1_j)(x_{1,N}+x_{j+2,N}+x_{j+3,N}), & 0\le j<k_1~\text{and}~\{p^1_j,p^1_{j+1}\}~\text{is nondashed} \\
    -\varphi(R^1_j)(x_{1,N}-x_{2,N}+x_{k_1+2,N}), & j=k_1,~\{p^1_0,p^1_{k_1}\}~\text{is dashed, and}~p^1_0<p_{k_1}^1 \\
    -\varphi(R^1_j)(x_{1,N}+x_{2,N}-x_{k_1+2,N}), & j=k_1,~\{p^1_0,p^1_{k_1}\}~\text{is dashed, and}~p^1_0>p_{k_1}^1 \\
    -\varphi(R^1_j)(x_{1,N}+x_{2,N}+x_{k_1+2,N}), & j=k_1~\text{and}~\{p^1_0,p^1_{k_1}\}~\text{is nondashed}
\end{cases}$$
for $0\le j\le k_1$; on the other hand, if $\mathcal{C}_1$ is a self-loop, then $$\mathbf{\widehat{E}_{-p_0^1,p_0^1}(\mathcal{P})}=-\varphi(E_{-p_0^1,p_0^1})(x_{1,N}+2x_{2,N}).$$ Similarly, if $\mathcal{C}_2$ is not a self-loop, then $$\mathbf{\widehat{R}}^2_j=\begin{cases}
    -\varphi(R^2_j)(x_{1,N}-x_{j+k_1+k_3+2,N}+x_{j+k_1+k_3+3,N}), & 0\le j<k_2,~\{p^2_j,p^2_{j+1}\}~\text{is dashed, and}~p^2_{j}<p^2_{j+1} \\
    -\varphi(R^2_j)(x_{1,N}+x_{j+k_1+k_3+2,N}-x_{j+k_1+k_3+3,N}), & 0\le j<k_2,~\{p^2_j,p^2_{j+1}\}~\text{is dashed, and}~p^2_{j}>p^2_{j+1} \\
    -\varphi(R^2_j)(x_{1,N}+x_{j+k_1+k_3+2,N}+x_{j+k_1+k_3+3,N}), & 0\le j<k_2~\text{and}~\{p^2_j,p^2_{j+1}\}~\text{is non-dashed} \\
    -\varphi(R^2_j)(x_{1,N}-x_{k_1+k_3+2,N}+x_{k_1+k_2+k_3+2,N}), & j=k_2,~\{p^2_0,p^2_{k_2}\}~\text{is dashed, and}~p^2_0<p_{k_2}^2 \\
    -\varphi(R^2_j)(x_{1,N}+x_{k_1+k_3+2,N}-x_{k_1+k_2+k_3+2,N}), & j=k_2,~\{p^2_0,p^2_{k_2}\}~\text{is dashed, and}~p^2_0>p_{k_2}^2 \\
    -\varphi(R^2_j)(x_{1,N}+x_{k_1+k_3+2,N}+x_{k_1+k_2+k_3+2,N}), & j=k_2~\text{and}~\{p^2_0,p^2_{k_2}\}~\text{is non-dashed}
\end{cases}$$ for $0\le j\le k_2$; on the other hand, if $\mathcal{C}_2$ is a self-loop, then $$\mathbf{\widehat{E}_{-p_0^2,p_0^2}(\mathcal{P})}=-\varphi(E_{-p_0^2,p_0^2})(x_{1,N}+2x_{k_1+k_3+2,N}).$$ Finally, with the above ordering of $\mathscr{B}_{C}(\mathcal{P})$, we have that $$\mathbf{\widehat{R}}^3_j=\begin{cases}
    -\varphi(R^3_j)(x_{1,N}-x_{2,N}+x_{k_1+3,N}), & j=0,~\{p_0^3,p^3_1\}~\text{is dashed, and}~p_0^3<p^3_1 \\
    -\varphi(R^3_j)(x_{1,N}+x_{2,N}-x_{k_1+3,N}), & j=0,~\{p_0^3,p^3_1\}~\text{is dashed, and}~p_0^3>p^3_1 \\
    -\varphi(R^3_j)(x_{1,N}+x_{2,N}+x_{k_1+3,N}), & j=0~\text{and}~\{p_0^3,p^3_{1}\}~\text{is non-dashed} \\
    -\varphi(R^3_j)(x_{1,N}-x_{j+k_1+2,N}+x_{j+k_1+3,N}), & 0< j<k_3,~\{p^3_j,p^3_{j+1}\}~\text{is dashed, and}~p^3_{j}<p^3_{j+1} \\
    -\varphi(R^3_j)(x_{1,N}+x_{j+k_1+2,N}-x_{j+k_1+3,N}), & 0< j<k_3,~\{p^3_j,p^3_{j+1}\}~\text{is dashed, and}~p^3_{j}>p^3_{j+1} \\
    -\varphi(R^3_j)(x_{1,N}+x_{j+k_1+2,N}+x_{j+k_1+3,N}), & 0< j<k_3~\text{and}~\{p^3_j,p^3_{j+1}\}~\text{is non-dashed} \\
\end{cases}$$ for $0\le j< k_3$. Note that Proposition~\ref{prop:noselfloops} covers the case where both $\mathcal{C}_1$ and $\mathcal{C}_2$ are self-loops. Considering Lemma~\ref{lem:nono}, there are five cases.
\bigskip

\noindent
\textbf{Case 1:} $\mathcal{C}_1$ or $\mathcal{C}_2$ is a self-loop and
\begin{itemize}
    \item[(1)] if $\mathcal{C}_1$ is a self-loop, then $\{p_{k_3-1}^3,p_k^3=p_0^2\}$ is non-dashed or $\{p_{k_3-1}^3,p_0^2\}$ is dashed and $p_{k_3-1}^3<p_0^2$, or
    \item[(2)] if $\mathcal{C}_2$ is a self-loop, then $\{p_0^3=p_0^1,p_1^3\}$ is non-dashed or $\{p_0^1,p_1^3\}$ is dashed and $p_0^1>p_1^3$.
\end{itemize}
Without loss of generality, assume that $\mathcal{C}_1$ is not a self-loop and $\mathcal{C}_2$ is a self-loop. Since $\varphi$ is a contact form, considering Remark~\ref{rem:nonzero}, we have that $\varphi(R^1_j)\neq 0$ for $0\le j\le k_1$, $\varphi(E_{-p_0^2,p_0^2})\neq 0$, and $\varphi(R^3_j)\neq 0$ for $0\le j< k_3$. Thus, we can define the collection of vectors $L^1_j=\frac{1}{-\varphi(R^1_j)}\mathbf{\widehat{R}}^1_j$ for $0\le j\le k_1$ and $L^3_j=\frac{1}{-\varphi(R^3_j)}\mathbf{\widehat{R}}^3_j$ for $0\le j< k_3$. Considering Lemma~\ref{lem:nono}, it is straightforward to verify that the collection of vectors $L_j^1$ for $0\le j\le k_1$ satisfies the hypotheses of Lemma~\ref{lem:conecyc}, and the collection of vectors $L_j^3$ for $0\le j< k_3$ satisfies the hypotheses of Lemma~\ref{lem:pathcon}. Since $k_1$ is even by assumption, applying Lemma~\ref{lem:conecyc} we find that there exist constants $c^1_j\in \{-1,1\}$ for $0\le j\le k_1$ such that $$\sum_{j=0}^{k_1}c^1_jL^1_j=x_{1,N}+2x_{2,N}.$$ If $k_3$ is odd, then applying Lemma~\ref{lem:pathcon} we find that there exist constants $c^3_j\in \{-1,1\}$ for $0\le j< k_3$ such that $$\sum_{j=0}^{k_3-1}c^3_jL^3_j=x_{1,N}+x_{2,N}+x_{k_1+k_3+2,N},$$ which implies that $$\frac{1}{2\varphi(E_{-p_0^2,p_0^2})}\mathbf{\widehat{E}_{-p_0^2,p_0^2}(\mathcal{P})}-\frac{1}{2}\sum_{j=0}^{k_1}c^1_jL^1_j+\sum_{j=0}^{k_3-1}c^3_jL^3_j=0.$$ On the other hand, if $k_3$ is even, then applying Lemma~\ref{lem:pathcon} we find that there exist constants $c^3_j\in \{-1,1\}$ for $0\le j< k_3$ such that $$\sum_{j=0}^{k_3-1}c^3_jL^3_j=x_{2,N}-x_{k_1+k_3+2,N},$$ which implies that $$\frac{1}{2\varphi(E_{-p_0^2,p_0^2})}\mathbf{\widehat{E}_{-p_0^2,p_0^2}(\mathcal{P})}+\frac{1}{2}\sum_{j=0}^{k_1}c^1_jL^1_j-\sum_{j=0}^{k_3-1}c^3_jL^3_j=0.$$ In either case, it follows that $\varphi\left(\widehat{C}(\mathfrak{g},\mathscr{B}(\mathfrak{g}))\right)$ does not have full rank, contradicting that $\varphi$ is a contact form. Therefore, $\mathfrak{g}_C(\mathcal{P})$ is not contact.
\bigskip

\noindent
\textbf{Case 2:} $\mathcal{C}_1$ or $\mathcal{C}_2$ is a self-loop and
\begin{itemize}
    \item[(1)] if $\mathcal{C}_1$ is a self-loop, then $\{p_{k_3-1}^3,p_{k_3}^3=p_0^2\}$ is dashed and $p_{k_3-1}^3>p_0^2$ or
    \item[(2)] if $\mathcal{C}_2$ is a self-loop, then $\{p_0^3=p_0^1,p_1^3\}$ is dashed and $p_0^1<p_1^3$.
\end{itemize}
Without loss of generality, assume that $\mathcal{C}_1$ is not a self-loop and $\mathcal{C}_2$ is a self-loop. As in Case 1, \\$\varphi(E_{-p_0^2,p_0^2})\neq0$ and we can define the collection of vectors $L^1_j=\frac{1}{-\varphi(R^1_j)}\mathbf{\widehat{R}}^1_j$ for $0\le j\le k_1$ and $L^3_j=\frac{1}{-\varphi(R^3_j)}\mathbf{\widehat{R}}^3_j$ for $0\le j< k_3$. Considering Lemma~\ref{lem:nono}, it is straightforward to verify that the collection of vectors $L_j^1$ for $0\le j\le k_1$ satisfies the hypotheses of Lemma~\ref{lem:conecyc2} and, if $k_3>1$, the collection of vectors $L_j^3$ for $1\le j< k_3$ satisfies the hypotheses of Lemma~\ref{lem:pathcon}. Since $k_1$ is even by assumption, applying Lemma~\ref{lem:conecyc2} we find that there exist constants $c^1_j\in \{-1,1\}$ for $0\le j\le k_1$ such that $$\sum_{j=0}^{k_1}c^1_jL^1_j=-x_{1,N}+2x_{2,N}.$$ If $k_3=1$, then we have that $$\frac{1}{2\varphi(E_{-p_0^2,p_0^2})}\mathbf{\widehat{E}_{-p_0^2,p_0^2}(\mathcal{P})}+\frac{1}{2}\sum_{j=0}^{k_1}c^1_jL^1_j+L_0^3=0.$$ If $k_3>1$ is even, then applying Lemma~\ref{lem:pathcon} we find that there exist constants $c^3_j\in \{-1,1\}$ for $1\le j< k_3$ such that $$\sum_{j=1}^{k_3-1}c^3_jL^3_j=x_{1,N}+x_{k_1+3,N}+x_{k_1+k_3+2,N},$$ which implies that $$2L_0^3+\sum_{j=0}^{k_1}c^1_jL^1_j-\frac{1}{\varphi(E_{-p_0^2,p_0^2})}\mathbf{\widehat{E}_{-p_0^2,p_0^2}(\mathcal{P})}-2\sum_{j=1}^{k_3-1}c^3_jL^3_j=0.$$ Finally, if $k_3>1$ is odd, then applying Lemma~\ref{lem:pathcon} we find that there exist constants $c^3_j\in \{-1,1\}$ for $1\le j< k_3$ such that $$\sum_{j=1}^{k_3-1}c^3_jL^3_j=x_{k_1+3,N}-x_{k_1+k_3+2,N},$$ which implies that $$2L_0^3+\sum_{j=0}^{k_1}c^1_jL^1_j+\frac{1}{\varphi(E_{-p_0^2,p_0^2})}\mathbf{\widehat{E}_{-p_0^2,p_0^2}(\mathcal{P})}-2\sum_{j=1}^{k_3-1}c^3_jL^3_j=0.$$ In all cases, it follows that $\varphi\left(\widehat{C}(\mathfrak{g},\mathscr{B}(\mathfrak{g}))\right)$ does not have full rank, contradicting that $\varphi$ is a contact form. Therefore, $\mathfrak{g}_C(\mathcal{P})$ is not contact.
\bigskip

\noindent
\textbf{Case 3:} Neither $\mathcal{C}_1$ nor $\mathcal{C}_2$ is a self-loop and one of the following holds:
\begin{itemize}
    \item[(1)] $\{p_0^3=p_0^1,p_1^3\}$ and $\{p_{k_3-1}^3,p_{k_3}^3=p_0^2\}$ are both non-dashed;
    \item[(2)] $\{p_0^3=p_0^1,p_1^3\}$ is non-dashed and $\{p_{k_3-1}^3,p_{k_3}^3=p_0^2\}$ is dashed with $p_{k_3-1}^3<p_0^2$;
    \item[(3)] $\{p_0^3=p_0^1,p_1^3\}$ is dashed with $p_0^1>p_1^3$ and $\{p_{k_3-1}^3,p_{k_3}^3=p_0^2\}$ is non-dashed; or
    \item[(4)] $\{p_0^3=p_0^1,p_1^3\}$ and $\{p_{k_3-1}^3,p_{k_3}^3=p_0^2\}$ are both dashed with $p_0^1>p_1^3$ and $p_{k_3-1}^3<p_0^2$.
\end{itemize}
Since $\varphi$ is a contact form, considering Remark~\ref{rem:nonzero}, we have that $\varphi(R^1_j)\neq 0$ for $0\le j\le k_1$, $\varphi(R^2_j)\neq 0$ for $0\le j\le k_2$, and $\varphi(R^3_j)\neq 0$ for $0\le j< k_3$. Thus, we can define the collection of vectors $L^1_j=\frac{1}{-\varphi(R^1_j)}\mathbf{\widehat{R}}^1_j$ for $0\le j\le k_1$, $L^2_j=\frac{1}{-\varphi(R^2_j)}\mathbf{\widehat{R}}^2_j$ for $0\le j\le k_2$, and $L^3_j=\frac{1}{-\varphi(R^3_j)}\mathbf{\widehat{R}}^3_j$ for $0\le j< k_3$. Considering Lemma~\ref{lem:nono}, it is straightforward to verify that, for $i=1,2$, the collection of vectors $L_j^i$ for $0\le j\le k_i$ satisfies the hypotheses of Lemma~\ref{lem:conecyc}. Moreover, the collection of vectors $L_j^3$ for $0\le j< k_3$ satisfies the hypotheses of Lemma~\ref{lem:pathcon}. Since $k_1$ and $k_2$ are even, applying Lemma~\ref{lem:conecyc} we find that there exist constants $c_j^1\in\{-1,1\}$ for $0\le j\le k_1$ and $c_j^2\in\{-1,1\}$ for $0\le j\le k_2$ such that $$\sum_{j=0}^{k_1}c_j^1L_j^1=x_{1,N}+2x_{2,N}$$ and $$\sum_{j=0}^{k_2}c_j^2L_j^2=x_{1,N}+2x_{k_1+k_3+2,N}.$$ If $k_3$ is odd, then applying Lemma~\ref{lem:pathcon} we find that there exist constants $c_j^3\in\{-1,1\}$ for $0\le j<k_3$ such that $$\sum_{j=0}^{k_3-1}c_j^3L_j^3=x_{1,N}+x_{2,N}+x_{k_1+k_{3}+2,N}$$ so that $$\sum_{j=0}^{k_1}c_j^1L_j^1+\sum_{j=0}^{k_2}c_j^2L_j^2-2\sum_{j=0}^{k_3-1}c_j^3L_j^3=0.$$ On the other hand, if $k_3$ is even, then applying Lemma~\ref{lem:pathcon} we find that there exist constants $c_j^3\in\{-1,1\}$ for $0\le j<k_3$ such that $$\sum_{j=0}^{k_3-1}c_j^3L_j^3=x_{2,N}-x_{k_1+k_{3}+2,N}$$ so that $$\sum_{j=0}^{k_1}c_j^1L_j^1-\sum_{j=0}^{k_2}c_j^2L_j^2-2\sum_{j=0}^{k_3-1}c_j^3L_j^3=0.$$ In either case, it follows that $\varphi\left(\widehat{C}(\mathfrak{g},\mathscr{B}(\mathfrak{g}))\right)$ does not have full rank, contradicting that $\varphi$ is a contact form. Therefore, $\mathfrak{g}_C(\mathcal{P})$ is not contact.
\bigskip

\noindent
\textbf{Case 4:} Neither $\mathcal{C}_1$ nor $\mathcal{C}_2$ is a self-loop and one of the following holds:
\begin{itemize}
    \item[(1)] $k_3=1$ and $\{p_0^3=p_0^1,p_1^3=p_0^2\}$ is dashed with $p_0^1>p_0^2$;
    \item[(2)] $\{p_0^3=p_0^1,p_1^3\}$ is non-dashed and $\{p_{k_3-1}^3,p_{k_3}^3=p_0^2\}$ is dashed with $p_{k_3-1}^3>p_0^2$;
    \item[(3)] $\{p_0^3=p_0^1,p_1^3\}$ is dashed with $p_0^1>p_1^3$ and $\{p_{k_3-1}^3,p_{k_3}^3=p_0^2\}$ is dashed with $p_{k_3-1}^3>p_0^2$; 
    \item[(4)] $k_3=1$ and $\{p_0^3=p_0^1,p_1^3=p_0^2\}$ is dashed with $p_0^1<p_1^3$;
    \item[(5)] $\{p_0^3=p_0^1,p_1^3\}$ is dashed with $p_0^1<p_1^3$ and $\{p_{k_3-1}^3,p_{k_3}^3=p_0^2\}$ is non-dashed;
    \item[(6)] $\{p_0^3=p_0^1,p_1^3\}$ is dashed with $p_0^1<p_1^3$ and $\{p_{k_3-1}^3,p_{k_3}^3=p_0^2\}$ is dashed with $p_{k_3-1}^3<p_0^2$.
\end{itemize}
Without loss of generality, assume that we are either in subcase (1), (2), or (3); the other subcases follow via a similar argument. Since $\varphi$ is a contact form, considering Remark~\ref{rem:nonzero} we have that $\varphi(R^1_j)\neq 0$ for $0\le j\le k_1$, $\varphi(R^2_j)\neq 0$ for $0\le j\le k_2$, and $\varphi(R^3_j)\neq 0$ for $0\le j< k_3$. Thus, we can define the collection of vectors $L^1_j=\frac{1}{-\varphi(R^1_j)}\mathbf{\widehat{R}}^1_j$ for $0\le j\le k_1$, $L^2_j=\frac{1}{-\varphi(R^2_j)}\mathbf{\widehat{R}}^2_j$ for $0\le j\le k_2$, and $L^3_j=\frac{1}{-\varphi(R^3_j)}\mathbf{\widehat{R}}^3_j$ for $0\le j< k_3$. Considering Lemma~\ref{lem:nono}, it is straightforward to verify that the collection of vectors $L_j^1$ for $0\le j\le k_1$ satisfies the hypotheses of Lemma~\ref{lem:conecyc}, the collection of vectors $L_j^2$ for $0\le j\le k_2$ satisfies the hypotheses of Lemma~\ref{lem:conecyc2}. Moreover, if $k_3>1$, then the collection of vectors $L_j^3$ for $0\le j< k_3$ satisfies the hypotheses of Lemma~\ref{lem:pathcon}. Since $k_1$ and $k_2$ are even, applying Lemmas~\ref{lem:conecyc} and~\ref{lem:conecyc2} we find that there exist constants $c_j^1\in\{-1,1\}$ for $0\le j\le k_1$ and $c_j^2\in\{-1,1\}$ for $0\le j\le k_2$ such that $$\sum_{j=0}^{k_1}c_j^1L_j^1=x_{1,N}+2x_{2,N}$$ and $$\sum_{j=0}^{k_2}c_j^2L_j^2=-x_{1,N}+2x_{k_1+k_3+2,N}.$$ If $k_3=1$, then we have that $$\sum_{j=0}^{k_1}c_j^1L_j^1-\sum_{j=0}^{k_2}c_j^2L_j^2-2L_0^3=0.$$ If $k_3>1$ odd, then applying Lemma~\ref{lem:pathcon} we find that there exist constants $c_j^3\in\{-1,1\}$ for $0\le j<k_3-1$ such that $$\sum_{j=0}^{k_3-2}c_j^3L_j^3=x_{2,N}-x_{k_1+k_3+1,N},$$ which implies that $$\sum_{j=0}^{k_1}c_j^1L_j^1-\sum_{j=0}^{k_2}c_j^2L_j^2-2\sum_{j=0}^{k_3-2}c_j^3L_j^3-2L_{k_3-1}^3=0.$$ If $k_3>1$ even, then applying Lemma~\ref{lem:pathcon} we find that there exist constants $c_j^3\in\{-1,1\}$ for $0\le j<k_3-1$ such that $$\sum_{j=0}^{k_3-2}c_j^3L_j^3=x_{1,N}+x_{2,N}+x_{k_1+k_3+1,N}$$ so that $$\sum_{j=0}^{k_1}c_j^1L_j^1+\sum_{j=0}^{k_2}c_j^2L_j^2-2\sum_{j=0}^{k_3-2}c_j^3L_j^3+2L_{k_3-1}^3=0.$$ In all cases, it follows that $\varphi\left(\widehat{C}(\mathfrak{g},\mathscr{B}(\mathfrak{g}))\right)$ does not have full rank, contradicting that $\varphi$ is a contact form. Therefore, $\mathfrak{g}_C(\mathcal{P})$ is not contact.
\bigskip

\noindent
\textbf{Case 5:} Neither $\mathcal{C}_1$ nor $\mathcal{C}_2$ is a self-loop, $\{p_0^3=p_0^1,p_1^3\}$ is dashed with $p_0^1<p_1^3$, and $\{p_{k_3-1}^3,p_{k_3}^3=p_0^2\}$ is dashed with $p_{k_3-1}^3>p_0^2$. Since $\varphi$ is a contact form, considering Remark~\ref{rem:nonzero} we have that $\varphi(R^1_j)\neq 0$ for $0\le j\le k_1$, $\varphi(R^2_j)\neq 0$ for $0\le j\le k_2$, and $\varphi(R^3_j)\neq 0$ for $0\le j< k_3$. Thus, we can define the collection of vectors $L^1_j=\frac{1}{-\varphi(R^1_j)}\mathbf{\widehat{R}}^1_j$ for $0\le j\le k_1$, $L^2_j=\frac{1}{-\varphi(R^2_j)}\mathbf{\widehat{R}}^2_j$ for $0\le j\le k_2$, and $L^3_j=\frac{1}{-\varphi(R^3_j)}\mathbf{\widehat{R}}^3_j$ for $0\le j< k_3$. Considering Lemma~\ref{lem:nono}, it is straightforward to verify that, for $i=1,2$, the collection of vectors $L_j^i$ for $0\le j\le k_i$ satisfies the hypotheses of Lemma~\ref{lem:conecyc2}. Moreover, if $k_3>2$, then the collection of vectors $L_j^3$ for $1\le j< k_3-1$ satisfies the hypotheses of Lemma~\ref{lem:pathcon}. Since $k_1$ and $k_2$ are even, applying Lemma~\ref{lem:conecyc} we find that there exist constants $c_j^1\in\{-1,1\}$ for $0\le j\le k_1$ and $c_j^2\in\{-1,1\}$ for $0\le j\le k_2$ such that $$\sum_{j=0}^{k_1}c_j^1L_j^1=-x_{1,N}+2x_{2,N}$$ and $$\sum_{j=0}^{k_2}c_j^2L_j^2=-x_{1,N}+2x_{k_1+k_3+2,N}.$$ If $k_3=2$, then we have that $$\sum_{j=0}^{k_1}c_j^1L_j^1+2L_0^2-\sum_{j=0}^{k_2}c_j^2L_j^2-2L_1^2=0.$$ If $k_3>2$ is odd, then applying Lemma~\ref{lem:pathcon} we find that there exist constants $c_j^3\in\{-1,1\}$ for $0<j<k_3-1$ such that $$\sum_{j=1}^{k_3-2}c_j^2L_j^2=x_{1,N}+x_{k_1+3,N}+x_{k_1+k_3+1,N},$$ which implies that $$\sum_{j=0}^{k_1}c_j^1L_j^1+2L_0^3+\sum_{j=0}^{k_2}c_j^2L_j^2+2L_{k_3-1}^3-2\sum_{j=1}^{k_3-2}c_j^2L_j^2=0.$$ Finally, if $k_3>2$ is even, then applying Lemma~\ref{lem:pathcon} we find that there exist constants $c_j^3\in\{-1,1\}$ for $0<j<k_3-1$ such that $$\sum_{j=1}^{k_3-2}c_j^2L_j^2=x_{k_1+3,N}-x_{k_1+k_3+1,N}$$ so that $$\sum_{j=0}^{k_1}c_j^1L_j^1+2L_0^3-\sum_{j=0}^{k_2}c_j^2L_j^2-2L_{k_3-1}^3-2\sum_{j=1}^{k_3-2}c_j^2L_j^2=0.$$ In all cases, it follows that $\varphi\left(\widehat{C}(\mathfrak{g},\mathscr{B}(\mathfrak{g}))\right)$ does not have full rank, contradicting that $\varphi$ is a contact form. Therefore, $\mathfrak{g}_C(\mathcal{P})$ is not contact.
\end{proof}

Thus, combining Propositions~\ref{prop:nocycle1},~\ref{prop:noevencycle},~\ref{prop:noselfloops},~\ref{prop:noocycle1},~\ref{prop:noocycle2}, and~\ref{prop:noocycle3} we obtain the following.

\begin{theorem}\label{thm:cycle}
Let $\mathcal{P}$ be a type-C poset of height one. If $\mathfrak{g}_C(\mathcal{P})$ is contact, then
\begin{enumerate}
    \item[\textup{(a)}] no connected component of $RG(\mathcal{P})$ contains an even cycle, more than one odd cycle, or more than one self-loop; and
    \item[\textup{(b)}] if $RG(\mathcal{P})$ is connected, then $RG(\mathcal{P})$ contains no cycles.
\end{enumerate}
\end{theorem}
\begin{proof}
(b) follows immediately from Propositions~\ref{prop:nocycle1},~\ref{prop:noevencycle},~\ref{prop:noselfloops},~\ref{prop:noocycle1},~\ref{prop:noocycle2}, and~\ref{prop:noocycle3}. As for (a), note that the arguments of Propositions~\ref{prop:noevencycle}--\ref{prop:noocycle3} still apply in the case where the poset $\mathcal{P}$ is disconnected and the cycles are contained in a single connected component of $RG(\mathcal{P})$.
\end{proof}

Consequently, if a connected, type-C poset $\mathcal{P}$ of height one is contact, then $RG(\mathcal{P})$ is necessarily a tree. In the following proposition, we show that this condition is also sufficient.
\bigskip

\begin{proposition}\label{prop:contacttree}
If $\mathcal{P}$ is a type-C poset of height one such that $RG(\mathcal{P})$ is a tree with $|V(\mathcal{P})|>1$, then $\mathfrak{g}_C(\mathcal{P})$ is contact. Moreover, if $p_0\in\mathcal{P}$ corresponds to a fixed vertex of degree one in $RG(\mathcal{P})$, $E_D$ denotes the collection of dashed edges of $RG(\mathcal{P})$, and $E_{\overline{D}}$ the collection of non-dashed edges, then $$\varphi=(D_{p_0})^*+\sum_{\{p,q\}\in E_{\overline{D}}(\mathcal{P})}(R^{\pm}_{p,q})^*+\sum_{\substack{\{p,q\}\in E_D(\mathcal{P})\\ p<q}}(R_{p,q})^*$$ is a contact form for $\mathfrak{g}_C(\mathcal{P})$.
\end{proposition}
\begin{proof}
Throughout, for $\mathcal{P}$ a type-C poset of height one such that $RG(\mathcal{P})$ is a tree with $|V(\mathcal{P})|>1$, we denote $\varphi(\widehat{C}(\mathfrak{g}_C(\mathcal{P}),\mathscr{B}_C(\mathcal{P})))$ by $\widehat{M}_{\varphi}(\mathcal{P})$. We show that $\det~\widehat{M}_{\varphi}(\mathcal{P})=1$ by induction on $|E(\mathcal{P})|$. If $|E(\mathcal{P})|=1$, then either
\begin{itemize}
    \item $\mathcal{P}=\{-2,-1,1,2\}$ with $-2\prec 1$ and $-1\prec 2$ or
    \item  $\mathcal{P}=\{-2,-1,1,2\}$ with $-2\prec -1$ and $1\prec 2$;
\end{itemize}
in either case, the claim can be checked directly. Assume that the result holds for $|E(\mathcal{P})|=n-1\ge 1$. Let $\mathcal{P}$ be a height-one, type-C poset such that $RG(\mathcal{P})$ is a tree with $|E(\mathcal{P})|=n>1$. Set $N=|V(\mathcal{P})|+|E(\mathcal{P})|+1$. Since $|E(\mathcal{P})|>1$, there exists a vertex $p\neq p_0$ such that $p$ has degree one in $RG(\mathcal{P})$, say $p$ is adjacent to $q$ in $RG(\mathcal{P})$ via a non-dashed edge; the dashed case follows via a similar argument. Removing vertex $p$ and the edge connecting $p$ and $q$ in $RG(\mathcal{P})$ results in $RG(\mathcal{P}')$, where $\mathcal{P}'$ is the poset induced by the subset $\mathcal{P}-\{-p,p\}\subset\mathcal{P}$. Note that $|E(\mathcal{P}')|=n-1$, so that our induction hypothesis applies to $\mathcal{P}'$ with $\varphi'=\varphi-(R^{\pm}_{p,q})^*$, i.e., $\det~\widehat{M}_{\varphi'}(\mathcal{P}')=1$. Since $p\neq p_0$ is a vertex of degree one in $RG(\mathcal{P})$, it follows that one obtains $\widehat{M}_{\varphi}(\mathcal{P})$ from $\widehat{M}_{\varphi'}(\mathcal{P}')$ by adjoining two new rows and columns corresponding to the elements of $\{R^{\pm}_{p,q},D_p\}=\mathscr{B}_C(\mathcal{P})\backslash \mathscr{B}_C(\mathcal{P}')$. Ordering $\mathscr{B}_C(\mathcal{P})$ so that $R^{\pm}_{p,q}$ is second to last and $D_p$ is last, we have that the last row of $\widehat{M}_{\varphi}(\mathcal{P})$, i.e., $\mathbf{\widehat{D}_p(\mathcal{P})}$, is equal to $x_{N-1,N}$; note that this implies that last column of $\widehat{M}_{\varphi}(\mathcal{P})$ is equal to the transpose of $-x_{N-1,N}$. Thus, computing $\det~\widehat{M}_{\varphi}(\mathcal{P})$ by first expanding along the last row followed by the last column we have $$\det~\widehat{M}_{\varphi}(\mathcal{P})=(-1)^{2N-1}(-1)^{2N-2}(-1)\det~\widehat{M}_{\varphi'}(\mathcal{P}')=(-1)^{4N-4}(1)=1.$$ The result follows.
\end{proof}
\bigskip

To finish the proof of Theorem~\ref{thm:contactchar}, we first require the following lemma.

\begin{lemma}\label{lem:decompose}
If $\mathcal{P}$ is a type-C poset of height one such that $RG(\mathcal{P})$ consists of connected components $\{K_1,\hdots,K_n\}$, then $$\mathfrak{g}_C(\mathcal{P})\cong \bigoplus_{i=1}^n\mathfrak{g}_C(\mathcal{P}_{K_i}),$$ where $\mathcal{P}_{K_i}$ is the unique type-C poset satisfying $RG(\mathcal{P}_{K_i})=K_i$.
\end{lemma}
\begin{proof}
Evidently, $\mathfrak{g}_C(\mathcal{P})\cong \bigoplus_{i=1}^n\mathfrak{g}_C(\mathcal{P}_{K_i})$ as vector spaces. Moreover, as non-trivial bracket relations can only exist between basis elements corresponding to vertices/edges in the same connected component of $RG(\mathcal{P})$, the result follows.
\end{proof}

We are now in a position to prove Theorem~\ref{thm:contactchar}.

\begin{proof}[Proof of Theorem~\ref{thm:contactchar}]
Assume that $RG(\mathcal{P})$ consists of the connected components $\{K_1,\hdots,K_n\}$ and let $\mathcal{P}_{K_i}$ denote the unique type-C poset such that $RG(\mathcal{P}_{K_i})=K_i$ for $1\le i\le n$. For the backward direction, assume that $K_1$ is the unique connected component which is a tree. Applying Proposition~\ref{prop:contacttree} and Theorem~\ref{thm:frob}, we find that $\mathcal{P}_{K_1}$ is contact and $\mathcal{P}_{K_i}$ is Frobenius for $2\le i\le n$. Now, if $\mathcal{P}$ is connected, then $\mathfrak{g}_C(\mathcal{P})=\mathfrak{g}_C(\mathcal{P}_{K_1})$, i.e., $\mathfrak{g}_C(\mathcal{P})$ is contact. On the other hand, if $\mathcal{P}$ is disconnected, then, applying Lemma~\ref{lem:decompose}, $\mathfrak{g}_C(\mathcal{P})$ is the direct sum of a contact Lie algebra with Frobenius Lie algebras, i.e., $\mathfrak{g}_C(\mathcal{P})$ is contact in this case as well.

For the forward direction, since a Lie algebra $\mathfrak{g}$ is contact only if $\ind\mathfrak{g}=1$, applying Theorem~\ref{thm:disjoint}, we find that $\ind\mathfrak{g}_C(\mathcal{P}_{K_i})=1$ for exactly one $1\le i\le n$ and $\ind\mathfrak{g}_C(\mathcal{P}_{K_j})=0$ for all other values of $1\le j\neq i\le n$. Without loss of generality, assume that $\ind\mathfrak{g}_C(\mathcal{P}_{K_1})=1$. Considering Theorems~\ref{thm:index1} and~\ref{thm:cycle} above, if $\mathfrak{g}(\mathcal{P})$ is contact, then $RG(\mathcal{P}_{K_1})$ must be a tree. Moreover, by Theorem~\ref{thm:frob} all other connected components must contain a single cycle consisting of an odd number of vertices. The result follows.
\end{proof}

\section{Directions for Further Research}\label{sec:ep}
The overall objective of this article is to continue the work initiated in \textbf{\cite{ContactLiePoset}} and to progress toward an eventual characterization of contact Lie poset algebras of classical type. For the interested reader, below we outline a few approaches that one may consider in pursuit of such a classification.

The approach that is simplest to describe, yet possibly the most cumbersome to execute, is a direct extension of the one used in this article. Specifically, it would be sufficient to extend the index formulas given in \textbf{\cite{seriesA}} and Section~\ref{sec:indf} so that they apply to type-A, B, C, and D Lie poset algebras associated with posets of arbitrary height and then apply similar arguments to those used here and in \textbf{\cite{ContactLiePoset}} to characterize those algebras that admit contact forms. With that said, extensive calculations by the authors suggest that a height-independent index formula for classical Lie poset algebras is out of reach, as the linear-algebraic techniques used here become more impractical as the corresponding posets grow in height.

An alternative approach to extending the techniques used here is to introduce a type-B, C, and D version of (contact) ``toral" posets. In \textbf{\cite{contacttoral}}, the authors extend the definition of ``toral" poset -- initially defined in \textbf{\cite{Binary}} -- to include posets corresponding to contact type-A Lie poset algebras, and they successfully construct contact forms for such Lie poset algebras. In short, contact ``toral" posets are constructed by identifying pairs of elements of ``building-block" posets together in a particular manner. The benefits to this approach lie in the generality of the definitions of (contact) ``toral-pairs" and (contact) ``toral" posets -- in particular, such definitions are height-independent -- and the combinatorial nature of the identification, or ``gluing," procedure. The drawback, however, is that it is currently unknown whether there are any contact type-A Lie poset algebras that are not ``toral." That is, while extending the notion of (contact) ``toral" poset to posets of types B, C, and D can possibly generate large families of such contact Lie poset algebras, it would be difficult to obtain a full characterization.

The least explored, yet perhaps most interesting, approach we propose here is via the Lie-algebraic concept of ``quasi-reductivity." Briefly, if $\mathfrak{g}$ is a complex Lie algebra of a connected linear algebraic group $G$ and has center $\mathfrak{z},$ then $\mathfrak{g}$ is \textit{quasi-reductive} if it admits a one-form $\varphi$ such that the center of $\faktor{\ker(d\varphi)}{\mathfrak{z}}$ consists of semisimple elements of $\mathfrak{g}$ (see \textbf{\cite{BM,duflo,PY}}). Such a one-form is said to be \textit{of reductive type}, and it can be shown that the contact form $\varphi$ given in Proposition~\ref{prop:contacttree} is of reductive type, i.e., each contact type-B, C, and D Lie poset algebra associated with a poset of height one is quasi-reductive. In fact, it is also straightforward to show that the contact forms constructed in the prequel \textbf{\cite{ContactLiePoset}} are of reductive type as well; thus, we claim that all contact Lie poset algebras of classical type associated with posets of height one are quasi-reductive. 
% A \textit{stable} Lie algebra $\mathfrak{g}$ is one that admits a one-form $\varphi$ for which there exists open neighborhood $V\subset\mathfrak{g}^*$ such that, for all $\psi\in V,$ the kernels $\ker(B_{\varphi})$ and $\ker(B_{\psi})$ are conjugate under the algebraic adjoint group of $\mathfrak{g}.$ 
On the other hand, an argument similar to that presented in Theorem 5 of \textbf{\cite{swclass}} proves that each index-one, quasi-reductive Lie poset algebra is contact. We are led to the following conjecture.
\begin{conj}\label{conj:contqr}
An index-one Lie poset algebra of type A, B, C, or D is contact if and only if it is quasi-reductive.
\end{conj}
\noindent Upon obtaining a proof of Conjecture~\ref{conj:contqr}, the characterization of contact Lie poset algebras could be acquired by investigating quasi-reductive Lie poset algebras, which, to the authors' knowledge, have not yet been identified. Furthermore, such a proof, in tandem with Theorem 5 of \textbf{\cite{swclass}}, would suggest a more general phenomenon occurring within the family of ``Lie proset algebras" (see \textbf{\cite{swclass}}).

% To the authors' knowledge, the identification of quasi-reductive Lie poset algebras has yet to be considered, although much work has been done on the classification of quasi-reductive ``seaweed" Lie algebras -- a family of combinatorially defined algebras closely related to Lie poset algebras (see \textbf{\cite{contactA,DK,Joseph2006}}).}

% \rn{That is, it can be shown that a height-one  if one could obtain a characterization of quasi-reductive Lie poset algebras -- which, to the author's knowledge, remains unexplored -- then the classification of contact Lie poset algebras would follow from restricting to index one.}
% , and further, that each contact Lie poset algebra of classical type is stable. Thus, if one could show that each stable, index-one Lie poset algebra of classical type is quasi-reductive, then \rn{the classification of contact Lie poset algebras would immediately follow from the classification of either the index-one quasi-reductive Lie poset algebras or the index-one stable Lie poset algebras.} Furthermore, a proof that stable Lie poset algebras are quasi-reductive would be an algebraic result of interest in its own right (see \textbf{\cite{sqpar}} and \textbf{\cite{PanRais}}), and it, in tandem with the results of Ammari (\textbf{\cite{sq,sqparex,sqpar}}) regarding ``seaweed" Lie algebras (see \textbf{\cite{contactA,DK,Joseph2006}}), would suggest a more general phenomenon occurring within the family of ``Lie proset algebras" (see \textbf{\cite{swclass}}).

%%%%%%%%%%%%%%%%%%%%%%%%%%%%%%%%%%%%%%

\end{document}